\documentclass[12pt]{article}

\usepackage[utf8]{inputenc} 
\usepackage[T1]{fontenc}          

\usepackage[largesc,osf]{newtxtext}
\usepackage[amsthm,nosymbolsc,vvarbb]{newtxmath}
\usepackage[cal=euler]{mathalpha}

\usepackage{array,cite,enumitem,mathtools}
\usepackage[dvipsnames]{xcolor}
\usepackage{graphicx}

\usepackage[margin=3cm]{geometry}   
\usepackage{titlesec}               
\usepackage{titletoc}               
\usepackage{url}                    

\usepackage[colorlinks]{hyperref}   
\hypersetup{linkcolor = blue, citecolor = green!60!black}

\title{Symmetry of some noncommutative sphere algebras}
       
\author{William J. Ugalde and Joseph C. Várilly%
\thanks{Email: \texttt{william.ugalde@ucr.ac.cr},\enspace
\texttt{joseph.varilly@ucr.ac.cr}}
\\[12pt]
{\small
Escuela de Matemática, Universidad de Costa Rica,
11501 San José, Costa Rica}}

\date{Nov 24, 2026}

\DeclareMathOperator{\diag}{diag} 
\DeclareMathOperator{\id}{\iota}  
\DeclareMathOperator{\tsum}{{\textstyle\sum}} 

\newcommand{\al}{\alpha}          
\newcommand{\Dl}{\Delta}          
\newcommand{\dl}{\delta}          
\newcommand{\eps}{\varepsilon}    
\newcommand{\om}{\omega}          
\newcommand{\Sg}{\Sigma}          

\newcommand{\bC}{{\mathbb{C}}}    
\newcommand{\bH}{\mathbb{H}}      
\newcommand{\bN}{\mathbb{N}}      
\newcommand{\bS}{\mathbb{S}}      
\newcommand{\bT}{\mathbb{T}}      

\newcommand{\sA}{\mathcal{A}}       
\newcommand{\sB}{\mathcal{B}}       
\newcommand{\sO}{\mathcal{O}}       

\newcommand{\pp}{\mathbf{p}}        
\newcommand{\xx}{\mathbf{x}}        
\newcommand{\zz}{\mathbf{z}}        

\newcommand{\rSp}{\mathrm{Sp}}      
\newcommand{\rSU}{\mathrm{SU}}      

\renewcommand{\geq}{\geqslant}      
\renewcommand{\leq}{\leqslant}      
\newcommand{\otto}{\leftrightarrow} 
\newcommand{\ovl}{\overline}        
\newcommand{\ox}{\otimes}           
\newcommand{\x}{\times}             
\newcommand{\8}{\bullet}            
\renewcommand{\.}{\cdot}            
\renewcommand{\:}{\colon}           

\newcommand{\dotox}{\mathrel{\dot\otimes}} 

\newcommand{\half}{{\mathchoice{\thalf}{\thalf}{\shalf}{\shalf}}} 
\newcommand{\quarter}{{\mathchoice{\tcuar}{\tcuar}{\scuar}{\scuar}}}
\newcommand{\scuar}{{\scriptstyle\frac{1}{4}}} 
\newcommand{\shalf}{{\scriptstyle\frac{1}{2}}} 
\newcommand{\tcuar}{\tfrac{1}{4}}   
\newcommand{\thalf}{\tfrac{1}{2}}   

\newcommand{\Adit}[1]{\ensuremath{(\mathrm{#1})}:\enspace}
\newcommand{\mot}[1]{\enspace\text{#1}\enspace} 
\newcommand{\set}[1]{\{\,#1\,\}}    
\newcommand{\word}[1]{\quad\text{#1}\quad} 

\newcommand{\twobytwo}[4]{\begin{pmatrix} 
   #1 & #2 \\ #3 & #4 \end{pmatrix}}

\titleformat{\section}{\normalfont\large\bfseries}
                      {\thesection}{1em}{}
\titlespacing{\section}{0pt}{*3.5}{*2.3}
\dottedcontents{section}[1pc]{\normalfont\bfseries}{1em}{40pc}
\titleformat{\subsection}{\normalfont\normalsize\bfseries}
                         {\thesubsection}{0.7em}{}
\titlespacing{\subsection}{0pt}{*3.25}{*1.5}
\dottedcontents{subsection}[3pc]{\normalfont}{2em}{40pc}
\titleformat{\subsubsection}{\normalfont\small\bfseries}
                         {\thesubsubsection}{0.5em}{}
\titlespacing{\subsubsection}{0pt}{*3.0}{*1.2}
\titleformat{\paragraph}[runin]{\normalfont\bfseries}{}{0pt}{}
\titlespacing{\paragraph}{0pt}{\medskipamount}{\wordsep}

\theoremstyle{plain}
\newtheorem{thm}{Theorem}[section]  
\newtheorem{lema}[thm]{Lemma}       
\newtheorem{corl}[thm]{Corollary}   
\newtheorem{prop}[thm]{Proposition} 

\theoremstyle{remark}
\newtheorem{remk}[thm]{Remark}      

\setlist[itemize]{label={$\diamond$}}
\setlist[enumerate]{label={\textup{(\alph*)}}, itemsep=0pt}

\begin{document}

\maketitle 

\begin{abstract}
Two known $q$-deformed (or `quantum') $7$-spheres, both denoted
$\bS^7_q$ in the literature, may be distinguished by the presence or
absence of symmetry under $\rSU_q(2)$. The quaternionic version of
$\bS^7_q$ has been shown by Brain and Landi to support such a
symmetry. Here we show that this is not the case for the older
$\bS^7_q$ introduced by Vaksman and Soibelman: and as a consequence,
these quantum $7$-spheres are not isomorphic.
\end{abstract}

\section{Introduction} 
\label{sec:intro}

Since the advent of quantum groups, several ``quantum homogeneous
spaces'' have been introduced, modeled as $q$-deformations of quotient
spaces of Lie groups by closed subgroups. The earliest example, the
standard Podleś sphere $\bS^2_q$ as a homogeneous space of the compact
quantum group $\rSU_q(2)$, is actually one of a family of $q$-deformed
$2$-spheres \cite{Dabrowski03, Podles87} that may be distinguished by
their sets of ``classical points''. By a compact quantum group $G_q$
one may understand either the $C^*$-bialgebra $C(G_q)$ with a
coinverse, or its unital Hopf $*$-algebra $\sA(G_q)$ of matrix
coefficients \cite[Chap.~1]{NeshveyevT14}. In this paper, we adopt the
latter point of view.

The term \textit{quantum homogeneous space} for $G_q$ then refers to a
$*$-algebra $\sA$ equipped with a right coaction
$\Psi \: \sA \to \sA \ox \sA(G_q)$. 

There are other ways of defining $q$-deformed spheres, of any
dimension $n \geq 2$. One can regard the $C^*$-algebra $C(\bS^n)$ as a
completion of an algebra with commuting selfadjoint generators
$x_0,x_1,\dots,x_n$ subject to the quadratic relation
$x_0^2 + x_1^2 +\cdots+ x_n^2 = 1$, or some equivalent. Given a real
parameter~$q$, a $q$-deformed $n$-sphere can be defined by a suitable
set of generators and relations, if one of the relations reduces to
the mentioned sphere relation when $q = 1$, together with several
(usually quadratic) commutation relations. An early instance is the
odd-dimensional sphere algebra $\sA(\bS^{2m+1}_q)$ of Vaksman and
Soibelman \cite{VaksmanS90}. An alternative approach is to define
$C(\bS^n_q)$ as a graph $C^*$-algebra, following Hong and Szymański
\cite{HongS02}, where $n$ can be either even or odd: in the odd case,
their $C(\bS^{2m+1}_q)$ is indeed a completion of Vaksman and
Soibelman's $\sA(\bS^{2m+1}_q)$. See also the recent review
\cite{DAndrea24}, with a simplified proof of these results
of~\cite{HongS02}.

A feature of either method is that the $q$-deformed sphere need not be
unique. For instance, the Podleś $2$-spheres form a family
$\bS^2_{qc}$ with an extra parameter~$c$, as already noted. Calow and
Matthes \cite{CalowM02} found a two-parameter family of $3$-spheres
$\bS^3_{pq}$; and so~on. A useful catalogue of these and other quantum
$2$-spheres and $3$-spheres was provided by Dąbrowski
\cite{Dabrowski03}. At least two $q$-deformed $4$-spheres have been
known for some time: one called $\bS^4_q$ was introduced
in~\cite{DabrowskiLM01} as a suspension of $\rSU_q(2)$; another,
designated $\Sg^4_q$, was obtained by Bonechi et~al
in~\cite{BonechiCT02} from a $q$-deformation of the Hopf bundle
$\bS^7 \to \bS^4$ whose total space has the coordinate algebra
$\sA(\bS^7_q)$ already mentioned.

In what follows, we consider $q$-deformations of odd-dimensional
spheres, particularly the $7$-sphere $\bS^7$. By lack of uniqueness,
the notation $\sA(\bS^{2m+1}_q)$ for their finitely generated algebras
is ambiguous. We shall reserve it for the algebras of Vaksman and
Soibelman, and denote by $\sO(\bS^7_q)$ the coordinate algebra of the
``symplectic'' $7$-sphere introduced by Landi, Pagani and
Reina~\cite{LandiPR06}. That sphere was further studied by Brain and
Landi in~\cite{BrainL12}, wherein the authors comment that these two
quantized versions of $\bS^7_q$ ``appear to be different''. Indeed
they are: our main result, Theorem~\ref{th:non-isomorphic} below, is
that $\sA(\bS^7_q)$ does not support a ``first-degree'' algebra
coaction of $\sA(\rSU_q(2))$ whereas (as is known) $\sO(\bS^7_q)$
does~so, and consequently the two algebras cannot be isomorphic.

To that end, we consider a general family of possible (right)
coactions of $\sA(\rSU_q(2))$ on the aforementioned sphere algebras,
to see which algebras support such coactions and which do not. The
classification of comodule algebras is of course a major undertaking;
here we accomplish only a small piece of~it. Working with generators
and relations, we only consider candidate coactions that are of first
degree in the generators. This excludes, for instance, the right
adjoint coaction of $\sA(\rSU_q(2))$ on itself, so it yields only a
partial classification; but it will be enough to answer the question
of Brain and Landi.

It should be noted that the $4$-sphere $\Sg^4_q$ in~\cite{BonechiCT02}
does appear as the invariant subalgebra of a coaction of the
\textit{coalgebra} $\sA(\rSU_q(2))$ on~$\sA(\bS^7_q)$, but at the
price of distorting its \textit{algebra} structure. This example does 
not gainsay our current no-go theorem.

It is worth mentioning that uniqueness results are indeed available
for the $C^*$-enveloping algebras. The work of Hong and Szymański
\cite{HongS02} showed that the enveloping algebras $C(\bS^{2m+1}_q)$
of $\sA(\bS^{2m+1}_q)$, for $0 \leq q < 1$, are graph $C^*$-algebras,
all isomorphic to one another; see also~\cite{DAndrea24}. The
$C^*$-enveloping algebras of the ``quaternionic'' algebras
$\sO(\bS^{2m+1}_q)$ of~\cite{LandiPR06} were studied by Saurabh
\cite{Saurabh17a, Saurabh17b}, who showed that, again for
$0 \leq q < 1$, they are all isomorphic to the Vaksman--Soibelman
$C^*$-algebra $C(\bS^{2m+1}_0)$, and consequently to each other. In
particular, our non-isomorphism result does not transfer to the
$C^*$-level.

One more item of note: the Vaksman--Soibelman sphere algebras
$\sA(\bS^{2m+1}_q)$, for fixed $m \in \bN$ and $0 \leq q < 1$, are
non-isomorphic for different~$q$; this was shown by D'Andrea
in~\cite{DAndrea26}. Whether the same is true for the algebras
$\sO(\bS^7_q)$ is still open.

\medskip

In Section~\ref{sec:the-setup} we briefly review the sphere algebras,
giving their generators and relations.
Section~\ref{sec:first-deg-coaction} introduces a family of linear
$*$-maps that should provide the coactions we seek. The following
Section~\ref{sec:coactions-on-VS} considers their existence (or~not)
on the Vaksman--Soibelman algebras $\sA(\bS^{2m+1}_q)$. The final
Section~\ref{sec:BL-reconsidered} poses the same question on the
quaternionic $q$-sphere $\sO(\bS^7_q)$. The (few) coactions found come
in families parametrized by $\bT = \set{\om \in \bC : |\om| = 1}$;
disregarding this inessential parameter, for $\sO(\bS^7_q)$ we recover
that of~\cite{BrainL12} and one~more.

A small issue arises with the coaction-invariant subalgebra in the
quaternionic case, which is identified in~\cite{BrainL12} as that of a
$q$-deformed $4$-sphere. However, we find that the relations among the
generators are somewhat different to those of~\cite{BrainL12}, as we
show in Appendix~\ref{app:coalgebra-invariants}.

In Appendix~\ref{app:no-SUtwo-coaction}, the analysis of
Section~\ref{sec:coactions-on-VS} is extended to show that a similar
coaction of the commutative Hopf algebra $\sA(\rSU(2))$ on the
Vaksman--Soibelman algebra $\sA(\bS^7_q)$ can also be ruled~out.

\medskip

One expects that a ``noncommutative sphere'' arising as an
$\sA(\rSU_q(2))$-comodule would be labeled by the same parameter~$q$.
But, mindful of experience \cite{LandiVS05} with Moyal-type
noncommutative spheres $\bS_\theta^n$, this need not be assumed
\textit{a~priori}. So we shall label the Vaksman--Soibelman spheres
with $p$ instead, allowing for the possibility $p \neq q$. In this 
article, we take $q$ to be real and positive, in the range
$0 < q < 1$ (except in Appendix~\ref{app:no-SUtwo-coaction}), and
likewise for~$p$.

\section{A few noncommutative spheres} 
\label{sec:the-setup}

We first recall the original ``quantum $3$-sphere'', namely the
quantum group $C(\rSU_q(2))$ introduced by Woronowicz
\cite{Woronowicz87}, where $0 < q < 1$. It may thought of as a
completion of the algebra of its matrix coefficients
$H = \sA(\rSU_q(2))$, which is the Hopf $*$-algebra generated by two
elements $a$ and~$b$, subject to the commutation rules:
\begin{gather}
ba = q ab,  \qquad  b^*a = qab^*, \qquad bb^* = b^*b,
\nonumber \\
a^*a + q^2 b^*b = 1,  \qquad  aa^* + bb^* = 1.
\label{eq:SUq2-relations} 
\end{gather}
The coproduct $\Dl \: H \to H \ox H$ is given by
\begin{alignat}{2}
\Dl(a) &:= a \ox a - q\,b \ox b^*, \qquad &
\Dl(a^*) &:= a^* \ox a^* - q\,b^* \ox b,
\nonumber \\
\Dl(b) &:= b \ox a^* + a \ox b, &
\Dl(b^*) &:= b^* \ox a + a^* \ox b^*,
\label{eq:SUq2-coproduct} 
\end{alignat}
with counit $\eps(a) = 1$, $\eps(b) = 0$, and antipode $S$ determined
by $S(a) = a^*$, $S(b) = -q b$, $S(b^*) = -q^{-1} b^*$, $S(a^*) = a$.
Here we follow the conventions of \cite{Dabrowski03, DabrowskiS03,
Naiad}. The relations \eqref{eq:SUq2-relations} also make sense with
$q > 1$, i.e., by exchanging $q \otto q^{-1}$, as many other authors
prefer.

The coproduct formulas \eqref{eq:SUq2-coproduct} can be handily
summarized \cite{Woronowicz87} as $\Dl(u) := u \dotox u$, acting on 
the ``fundamental unitary'' 
$$
u = \twobytwo{a}{b}{-q b^*}{a^*}.
$$
Another coproduct on the same algebra \eqref{eq:SUq2-relations} is
found by replacing $u$ by its inverse $u^{-1} = u^*$; then
$\Dl'(u^*) := u^* \dotox u^*$ gives rise to a different coalgebra
structure; that convention is adopted in the recent monograph
\cite{KaadK25}, for instance.

There is an obvious action of the circle group $\bT$ on
$\sA(\rSU_q(2))$ by automorphisms, given on generators by
$a \mapsto a$, $b \mapsto \om b$ with $|\om| = 1$, which is compatible
with the coproduct \eqref{eq:SUq2-coproduct}.

\medskip

The earliest ``quantum $7$-sphere'' was developed by Vaksman and
Soibelman \cite{VaksmanS90}, by identifying an appropriate subalgebra
of $\sA(\rSU_p(4))$, where $0 < p < 1$, with four $*$-algebra
generators $z_0,z_1,z_2,z_3$, satisfying the commutation relations:
\begin{subequations}
\label{eq:VS-odd} 
\begin{align}
z_j z_i &= p z_i z_j \word{for} i < j,
\label{eq:VS-odd-first} 
\\
z_i^* z_j &= p z_j z_i^* \word{for} i \neq j,
\label{eq:VS-odd-second} 
\\
z_i^* z_i &= z_i z_i^* + (1 - p^2) \sum_{j>i} z_j z_j^*
\label{eq:VS-odd-third} 
\end{align}
\end{subequations}
and the spherical condition
\begin{subequations}
\label{eq:VS-sphere} 
\begin{equation}
z_0 z_0^* + z_1 z_1^* + z_2 z_2^* + z_3 z_3^* = 1.
\label{eq:VS-seven-sphere} 
\end{equation}

(Again, the relations \eqref{eq:VS-odd} also make sense for $p > 1$.)
When $p = 1$, they reduce to the commutative coordinate algebra
of~$\bC^4$, and the last relation determines $\bS^7 \subset \bC^4$.

Actually, $\sA(\bS^7_p)$ is a special case of a family of algebras
$\sA(\bS^{2m+1}_p)$ with generators $z_0,z_1,\dots,z_m$, for any
$m = 2,3,\dots$, subject to \eqref{eq:VS-odd} and the more general
sphere relation:
\begin{equation}
z_0 z_0^* + z_1 z_1^* +\cdots+ z_m z_m^* = 1.
\label{eq:VS-odd-sphere} 
\end{equation}
\end{subequations}
This is obtained as a subalgebra of $\sA(\rSU_p(m+2))$ in the same
way. An excellent recent survey of the properties of such ``odd
quantum spheres'' is~\cite{MikkelsenK23}; see also \cite{DAndrea26}
and~\cite{KlimykS97}.

\begin{remk} 
\label{rk:independent-quadratics}
There are no quadratic relations among the generators, other than
those given by \eqref{eq:VS-odd} and \eqref{eq:VS-odd-sphere}. In
particular, the set of monomials
$\set{z_i z_j, z_k z^*_l, z^*_r z^*_s : i \leq j,\ k < l,\ r \leq s}$
is linearly independent.
\end{remk}

Note especially that \eqref{eq:SUq2-relations} is a particular case of
\eqref{eq:VS-odd} and \eqref{eq:VS-odd-sphere}, for $p = q$ and
$m = 1$, taking $z_0 = a$, $z_1 = b$. Thus
$\sA(\bS^3_q) = \sA(\rSU_q(2))$.

\medskip

On the other hand, $\bS^7$ is also the unit sphere in the quaternionic
plane~$\bH^2$. In that picture, the compact symmetry group $\rSU(4)$
of $\bC^4$ is replaced by the compact symplectic group $\rSp(2)$ of
$\bH^2$, that also respects the quaternionic structure.

In~\cite{LandiPR06}, the coordinate algebra of $\rSp(2)$ was modified
to that of its $q$-deformed compact quantum group $\rSp_q(2)$. A
suitable $q$-deformed $7$-sphere is then given by the $*$-algebra
$\sO(\bS^7_q)$, with four generators $x_0,x_1,x_2,x_3$, organized 
in two pairs $x_0,x_1$ and $x_2,x_3$.

Here we follow the slightly different conventions in the successor
paper \cite{BrainL12} by Brain and Landi, but with $q \otto q^{-1}$,
since \cite{BrainL12} follows the alternative convention for the
$\rSU_q(2)$ generators. First, we list the relations internal to
both pairs:
\begin{subequations}
\label{eq:BL-S7-commutation} 
\begin{alignat}{4}
x_1 x_0 &= q^{-1} x_0 x_1, \qquad & 
x_1^* x_0 &= q x_0 x_1^*,  \qquad &
x_0^* x_0 &= x_0 x_0^*, \qquad &
x_1^* x_1 &= x_1 x_1^* + (1 - q^2) x_0 x_0^*, 
\nonumber \\
x_3 x_2 &= q^{-1} x_2 x_3, \qquad &
x_3^* x_2 &= q x_2 x_3^*, \qquad &
x_2^* x_2 &= x_2 x_2^*, \qquad &
x_3^* x_3 &= x_3 x_3^* + (1 - q^2) x_2 x_2^*.
\label{eq:BL-S7-quaternionic} 
\end{alignat}
There are also ``braided'' relations, mixing both pairs:
\begin{alignat}{2}
x_2 x_0 &= q^{1/2} x_0 x_2, \qquad & 
x_3 x_0 &= q^{-1/2} x_0 x_3,
\nonumber \\ 
x_2 x_1 &= q^{1/2} x_1 x_2,  \qquad &
x_3 x_1 &= q^{-1/2} x_1 x_3,
\nonumber \\ 
x_2^* x_0 &= q^{1/2} x_0 x_2^*, \qquad & 
x_3^* x_0 &= q^{-1/2} x_0 x_3^* - q^{-1}(1 - q^2) x_2 x_1^*,
\nonumber \\
x_2^* x_1 &= q^{1/2} x_1 x_2^*, \qquad &
x_3^* x_1 &= q^{-1/2} x_1 x_3^* + (1 - q^2) x_2 x_0^*.
\label{eq:BL-S7-braided} 
\end{alignat}
\end{subequations}
And finally, a spherical relation, that can be written in two 
equivalent forms:
\begin{align}
x_0 x_0^* + x_1 x_1^* + x_2 x_2^* + x_3 x_3^* &= 1,
\label{eq:BL-sphere} 
\\
\text{or} \quad
q^2 x_0^* x_0 + x_1^* x_1 + q^2 x_2^* x_2 + x_3^* x_3 &= 1.
\nonumber
\end{align}

It is obvious that, for $p = q$, there is no simple correspondence
between the generators $z_i$ of $\sA(\bS^7_q)$ and $x_i$ of
$\sO(\bS^7_q)$ that would yield an algebra isomorphism. But that is
not to say that no such isomorphism is possible. The lack of an
isomorphism will eventually follow from the symmetry properties of
these algebras, not from a naïve matching of generators or lack
thereof.

\section{Algebra coactions of first degree} 
\label{sec:first-deg-coaction}

We denote by $\Psi \: \sA \to \sA \ox H$ a 
right coaction of a Hopf algebra $H$ on a $*$-algebra~$\sA$:
$$
(\Psi \ox \id)\,\Psi = (\id \ox \Dl)\,\Psi \: \sA \to \sA \ox H \ox H
\word{and} (\id \ox \eps)\,\Psi = \id \: \sA \to \sA
$$
where $\Psi$ is a linear $*$-map, i.e., $\Psi(z^*) = \Psi(z)^*$ for
$z \in \sA$.

We consider any such coaction $\Psi$ of $H = \sA(\rSU_q(2))$ on a
finitely generated $*$-algebra $\sA$ that is of first-degree on the
generators of $\sA$ and~$H$. Explicitly,
\begin{align}
\Psi(z_j) &= \sum_k (a_{jk} z_k \ox a + a'_{jk} z^*_k \ox a 
+ b_{jk} z_k \ox b + b'_{jk} z^*_k \ox b
\nonumber \\
&\qquad + c_{jk} z_k \ox a^* + c'_{jk} z^*_k \ox a^* 
+ d_{jk} z_k \ox b^* + d'_{jk} z^*_k \ox b^*).
\label{eq:first-degree-Psi} 
\end{align}
These coefficients are organized as matrices $A = [a_{jk}]$,
$B = [b_{jk}]$, etc., with rows and columns indexed by
$\{0,1,\dots,m\}$. Collecting the generators as columns
$\zz = (z_0,\dots,z_m)^t$ and $\zz^* = (z^*_0,\dots,z^*_m)^t$, these
relations and their complex conjugates can be rewritten as
\begin{align}
\Psi(\zz) &= (A\zz + A'\zz^*) \ox a + (B\zz + B'\zz^*) \ox b
+ (C\zz + C'\zz^*) \ox a^* + (D\zz + D'\zz^*) \ox b^*,
\nonumber \\
\Psi(\zz^*) &= (\ovl{A'}\zz + \ovl{A}\zz^*) \ox a^* \!
+ (\ovl{B'}\zz + \ovl{B}\zz^*) \ox b^* 
+ (\ovl{C'}\zz + \ovl{C}\zz^*) \ox a
+ (\ovl{D'}\zz + \ovl{D}\zz^*) \ox b.
\label{eq:general-Psi} 
\end{align}
It is easily seen that the coefficient matrices satisfy
\begin{equation}
A + C = I,  \qquad  A' + C' = 0,
\label{eq:eliminate-Cs} 
\end{equation}
by applying $(\id \ox \eps) \circ \Psi = \id$, using
$\eps(a) = \eps(a^*) = 1$ and $\eps(b) = \eps(b^*) = 0$.

\medskip

The coaction relation $(\Psi \ox \id)\,\Psi = (\id \ox \Dl)\,\Psi$
reduces to several equations for the coefficient matrices, that can be
sorted according to the monomials of $H \ox H$ that appear, on
combining \eqref{eq:SUq2-coproduct} with \eqref{eq:general-Psi} and
\eqref{eq:eliminate-Cs}. For instance, terms of type $- \ox a \ox a$
appear as follows:
\begin{align*}
(\Psi \ox \id) \Psi(\zz)
&= (\Psi \ox \id) \bigl( (A\zz + A'\zz^*) \ox a +\cdots \bigr)
= \bigl( A \Psi(\zz) + A'\Psi(\zz^*) \bigr) \ox a +\cdots
\\
&= \bigl( A(A\zz + A'\zz^*) + A'(\ovl{C'}\zz + \ovl{C}\zz^*) \bigr)
\ox a \ox a +\cdots
\\
\text{and} \quad (\id \ox \Dl) \Psi(\zz)
&= (A\zz + A'\zz^*) \ox \Dl(a) +\cdots
= (A\zz + A'\zz^*) \ox a \ox a +\cdots
\end{align*}
Comparing coefficients of $\zz \ox a \ox a$ and $\zz^* \ox a \ox a$
and using $C' = - A'$, we get two matrix equations:
$$
A^2 - A' \ovl{A'} = A,  \qquad AA' + A'\ovl{C} = A'.
$$
Terms of type $- \ox a^* \ox a^*$ give the same relations, since both
$\Psi$ and $\Dl$ are $*$-maps.

Table~\ref{tb:coaction-coefficients} gives the matrix relations so
obtained, after eliminating some redundancies.

\begin{table}[htb] 
$$
\begin{array}{@{}r||r@{\;}l|r@{\;}l@{}}
\hline
- \ox a \ox a & A^2 - A'\ovl{A'} &= A &  AA' + A'\ovl{C} &= A'
\\
- \ox b \ox a & AB + A'\ovl{D}' &= 0 &  AB' + A'\ovl{D} &= 0
\\
- \ox b^* \ox a &
AD + A' \ovl{B'} &= D &  AD' + A' \ovl{B} &= D'
\\
- \ox a \ox b & BA - B'\ovl{A'} &= B &  BA' + B'\ovl{C} &= B' 
\\
- \ox b \ox b & B^2 + B'\ovl{D'} &= 0 &  BB' + B'\ovl{D} &= 0
\\
- \ox b^* \ox b & BD + B'\ovl{B'} &= -qC &  BD' + B'\ovl{B} &= qA'
\\
- \ox a \ox b^* & DA - D'\ovl{A'} &= 0 &  DA' + D'\ovl{C} &= 0
\\
- \ox b \ox b^* & DB + D'\ovl{D'} &= -qA &  DB' + D'\ovl{D} &= -qA'
\\
- \ox b^* \ox b^* & D^2 + D'\ovl{B'} &= 0 &  DD' + D'\ovl{B} &= 0
\\ \hline
\end{array}
$$
\caption{Coaction coefficients}
\label{tb:coaction-coefficients} 
\end{table}

\begin{remk} 
\label{rk:SUq-two-coproduct}
It is easily checked that for $\sA = H = \sA(SU_q(2))$, the coproduct
$\Dl$ of \eqref{eq:SUq2-coproduct} satisfies these relations, where
\begin{equation}
A = \twobytwo{1}{0}{0}{0}, \quad 
B = \twobytwo{0}{0}{1}{0}, \quad 
C = \twobytwo{0}{0}{0}{1}, \quad 
D = \twobytwo{0}{-q}{0}{0},
\label{eq:SUq-two-coproduct} 
\end{equation}
and $A' = B' = C' = D' = 0$.
\end{remk}

\begin{remk} 
\label{rk:BL-coaction}
In \cite{BrainL12}, a right coaction 
$\dl_R \: \sO(\bS^7_q) \to \sO(\bS^7_q) \ox \sA(SU_q(2))$ is obtained,
given explicitly by
\begin{alignat}{2}
\dl_R(x_i) &:= x_i \ox a + x_{i+1}^* \ox b^* \quad &
\text{for } i &= 0,2,
\nonumber \\
\dl_R(x_j) &:= x_j \ox a - q x_{j-1}^* \ox b^* \quad &
\text{for } j &= 1,3,
\label{eq:BL-coaction} 
\end{alignat}
which is clearly of type~\eqref{eq:general-Psi}, with
$B = B' = C = C' = D = 0$, $A = I_4$, and 
$$
D' = \begin{pmatrix}
0 & 1 & 0 & 0 \\
-q & 0 & 0 & 0 \\
0 & 0 & 0 & 1 \\
0 & 0 & -q & 0 \end{pmatrix}.
$$
\end{remk}

We now impose the requirement that the linear $*$-map $\Psi$ of
\eqref{eq:general-Psi} is an algebra homomorphism, i.e., that $\Psi$
is a \emph{$*$-algebra coaction}. The formulas \eqref{eq:general-Psi}
then impose many quadratic relations among the coefficient matrix
entries.

At first, we can treat two (families of) algebras in parallel, namely
the Vaksman--Soibelman algebra $\sA(\bS^{2m+1}_p)$ with $0 < p < 1$;
and the Brain--Landi algebra $\sO(\bS^7_q)$ with $p = q$ and $m = 3$.

\begin{lema} 
\label{lm:bare-columns}
The commutations relations of either algebra imply that each column of
the matrices $A,B,C,D,A',C',B',D'$ contains at most one nonzero entry.

Moreover, the matrices $A$ and $C = I - A$ are diagonal.
\end{lema}

\begin{proof}
From the relation \eqref{eq:VS-odd-first} in $\sA(\bS^{2m+1}_p)$, we
get the conditions:
\begin{equation}
z_j z_i = p z_i z_j  \implies 
\Psi(z_j) \Psi(z_i) = p \Psi(z_i) \Psi(z_j) \word{for} i < j.
\label{eq:Psi-on-VS-first} 
\end{equation}
The commutation relations \eqref{eq:BL-S7-commutation} yield analogous
formulas for a coaction on $\sO(\bS^7_q)$.

Comparisons of coefficients of $z_k^2 \ox u^2$ and $z_k^{*2} \ox u^2$
in these formulas, for $u = a,b,a^*,b^*$, lead to the equalities:
$$
a_{ik} a_{jk} = b_{ik} b_{jk} 
= c_{ik} c_{jk} = d_{ik} d_{jk} = 0,  \quad
a'_{ik} a'_{jk} = b'_{ik} b'_{jk} 
= c'_{ik} c'_{jk} = d'_{ik} d'_{jk} = 0
$$
for $i \neq j$; here $k$ is any column index. The first assertion 
follows at once.

If $A$ were not a diagonal matrix, say $a_{01} \neq 0$, then
$a_{11} = 0$, so $c_{01} = - a_{01} \neq 0$ and 
$c_{11} = 1 - a_{11} = 1$, so $C$ would have a column with at least
two nonzero entries, which is already ruled out.
\end{proof}

\begin{lema} 
\label{lm:ones-and-zeros}
The diagonal matrices $A$ and $C = I - A$ have only $0$ and~$1$ as
diagonal entries, in complementary positions.
\end{lema}

\begin{proof}
Apply $\Psi$ to the relation \eqref{eq:VS-odd-sphere} of
$\sA(\bS^{2m+1}_p)$. Examining the resulting terms with second tensor
factor~$a^2$, we see that
\begin{align*}
a_{00} \bar c_{00}\, z_0 z_0^* + a_{11} \bar c_{11}\, z_1 z_1^*
+\cdots+ a_{mm} \bar c_{mm}\, z_m z_m^* = 0.
\end{align*}
An analogous equation comes from applying $\Psi$ to the relation
\eqref{eq:BL-sphere} of $\sO(\bS^7_q)$.

Since $c_{ii} = 1 - a_{ii}$, linear independence of the $z_i z_i^*$
entails $a_{ii} = 0$ or~$1$ for each~$i$; and then $c_{ii} = 1$ or~$0$,
respectively.
\end{proof}

\section{Coactions on the Vaksman--Soibelman spheres} 
\label{sec:coactions-on-VS}

We now consider properties of algebra coactions that are specific to 
the Vaksman--Soibelman algebras. In the rest of this section, 
we take $\sA = \sA(\bS^{2m+1}_p)$ with $0 < p < 1$.

As well as~\eqref{eq:Psi-on-VS-first}, we now consider the relations
\eqref{eq:VS-odd-second}, namely, $z_i^* z_j = p z_j z_i^*$ for
$i \neq j$ in $\sA(\bS^{2m+1}_p)$. Applying $\Psi$ to these relations,
we get more conditions on the coaction:
\begin{equation}
\Psi(z_j^*) \Psi(z_i) = p\,\Psi(z_i) \Psi(z_j^*) \mot{for} i \neq j.
\label{eq:Psi-on-VS-second} 
\end{equation}

\begin{lema} 
\label{lm:A-prime-vanishes}
If $\Psi$ is an algebra coaction of $\sA(\rSU_q(2))$ on $\sA$, then
$A' = C' = 0$.
\end{lema}

\begin{proof}
From \eqref{eq:eliminate-Cs}, it is enough to prove that $A' = 0$. 
First we show that $A'$ is diagonal. 

Using the relations \eqref{eq:general-Psi} and
again comparing coefficients of the terms $z_k^2 \ox a^2$ and 
$z_k^{*2} \ox a^2$ -- for fixed~$k$ -- in the relations
\eqref{eq:Psi-on-VS-second}, we find that
$$
\bar c'_{jk} a_{ik} = p a_{ik} \bar c'_{jk} \word{and}
\bar c_{jk} a'_{ik} = p a'_{ik} \bar c_{jk} \word{for} i \neq j.
$$
Since $C' = -A'$, these relations imply that
$$
a_{ik} \bar a'_{jk} = 0  \word{and} c_{ik} \bar a'_{jk} = 0
\word{for} i \neq j.
$$
Since $A + C = I$ and $A$ is diagonal, that makes $a'_{ji} = 0$ for
$i \neq j$, i.e., $A'$ is diagonal. Lemma~\ref{lm:ones-and-zeros}
also implies $A^2 = A$, so the first row of
Table~\ref{tb:coaction-coefficients} yields $A' \ovl{A'} = 0$; thus,
all diagonal entries of $A'$ also vanish.
\end{proof}

For the algebra $\sA(\bS^{2m+1}_p)$, the relations
\eqref{eq:general-Psi} reduce to:
\begin{align}
\Psi(z_i) &= a_{ii} z_i \ox a + c_{ii} z_i \ox a^*
+ \tsum_k \bigl( b_{ik} z_k + b'_{ik} z^*_k \bigr) \ox b
+ \tsum_k \bigl( d_{ik} z_k + d'_{ik} z^*_k \bigr) \ox b^*,
\nonumber \\
\Psi(z_j^*) &= a_{jj} z_j^* \ox a^* + c_{jj} z_j^* \ox a
+ \tsum_l \bigl( \bar b_{il} z_l^* + \bar b'_{il} z_l \bigr) \ox b^*
+ \tsum_l \bigl( \bar d_{il} z_l^* + \bar d'_{il} z_l \bigr) \ox b,
\label{eq:better-Psi} 
\end{align}
and the previous Table~\ref{tb:coaction-coefficients} can now be
simplified into the nearby Table~\ref{tb:simpler-coefficients}.

\begin{table}[htb] 
$$
\begin{array}{@{}r||r@{\;}l|r@{\;}l@{}}
\hline
\mathrm{(a)} &  A^2 &= A,
\\
\mathrm{(b)} &  AB &= 0,  &  AB' &= 0,
\\
\mathrm{(c)} &  CD &= 0,  &  CD' &= 0,
\\
\mathrm{(d)} &  BC &= 0,  &  B'A &= 0,
\\
\mathrm{(e)} &  B^2 + B'\ovl{D'} &= 0,  & BB' + B' \ovl{D} &= 0,
\\
\mathrm{(f)} &  BD + B'\ovl{B'} &= -qC, & BD' + B'\ovl{B} &= 0,
\\
\mathrm{(g)} &  DA &= 0,  &  D'C &= 0,
\\
\mathrm{(h)} &  DB + D'\ovl{D'} &= -qA, & DB' + D' \ovl{D} &= 0,
\\
\mathrm{(i)} &  D^2 + D'\ovl{B'} &= 0,  & DD' + D' \ovl{B} &= 0.
\\ \hline
\end{array}
$$
\caption{Coaction coefficients for the V--S algebra}
\label{tb:simpler-coefficients} 
\end{table}

This provides some information on the rows and columns of $B$, $B'$,
$D$ and~$D'$, according as the diagonal entries of $A$ are~$1$~or~$0$.
Denote by $b_{k\8}$ the $k$-th row of~$B$ and by $b_{\8l}$ the
$l$-th column of~$B$; and similarly for the other coefficient
matrices.

\begin{lema} 
\label{lm:scalar-A}
In $\sA = \sA(\bS^{2m+1}_p)$, if $a_{ii} = 1$ and $a_{jj} = 0$ for some
pair of indices $i \neq j$, we deduce from rows (b), (c), (d) and~(g) 
of Table~\ref{tb:simpler-coefficients} that:
\begin{alignat}{3}
AB &= 0 &\mot{and} AB' &= 0, &
\mot{implying} b_{i\8} &= b'_{i\8} = 0,
\nonumber \\
BC &= 0 &\mot{and} B'A &= 0, &
\mot{implying} b_{\8j} &= b'_{\8i} = 0,
\nonumber \\
CD &= 0 &\mot{and} CD' &= 0, &
\mot{implying} d_{j\8} &= d'_{j\8} = 0,
\nonumber \\
DA &= 0 &\mot{and} D'C &= 0, &
\mot{implying} d_{\8i} &= d'_{\8j} = 0.
\label{eq:flensing} 
\end{alignat}
\end{lema}

\begin{proof}
It is enough to recall that $A + C = 1$, so $c_{ii} = 0$ and 
$c_{jj} = 1$.
\end{proof}

\begin{corl} 
\label{cr:scalar-A}
If $A = I$, then $B = B' = C = D = 0$, $D'\ovl{D'} = -qI$, and
$\diag(D') = (0,\dots,0)$.

If $C = I$, then $A = B = D = D' = 0$, $B'\ovl{B'} = -qI$, and
$\diag(B') = (0,\dots,0)$.
\end{corl}

\begin{proof}
If $A = I$, then $C = 0$ by \eqref{eq:eliminate-Cs}; and
$B = B' = D = 0$ follow from rows
(b) and~(g) of Table~\ref{tb:simpler-coefficients}; row~(h) provides
$D'\ovl{D'} = -qI$. Next, if some $d'_{rr} \neq 0$,
then this is the only nonzero entry in the column~$d'_{\8r}$,
by Lemma~\ref{lm:bare-columns}. The $r$-th diagonal entry of
$D'\ovl{D'}$ would then be $|d'_{rr}|^2 = -q$, which is absurd.
Therefore, $d'_{rr} = 0$ for every~$r$.

Similarly, if $C = I$, then $A = 0$ and rows (c) and~(d) of the table
yield $B = D = D' = 0$, and row~(f) provides $B'\ovl{B'} = -qI$. The
same argument as before gives a zero diagonal of~$B'$.
\end{proof}

We next examine several other possibilities for the quantities of ones
and zeros in the diagonal entries of the matrices $A$ and~$C$. It is
convenient to deal with $3$-spheres separately, so for a while we
shall assume that $m > 1$.

\begin{prop} 
\label{pr:scalar-A}
If $m > 1$, there is no algebra coaction $\Psi$ of $\sA(\rSU_q(2))$ on
$\sA(\bS^{2m+1}_p)$ with either $A = I$ or $A = 0$.
\end{prop}

\begin{proof}
Assume $A = I$. Corollary~\ref{cr:scalar-A} then simplifies
\eqref{eq:general-Psi} to
\begin{equation}
\Psi(z_i) = z_i \ox a + \tsum_k d'_{ik} z^*_k \ox b^*  \word{and}
\Psi(z_j^*) = z_j^* \ox a^* + \tsum_l \bar d'_{jl} z_l \ox b,
\label{eq:scalar-A} 
\end{equation}
with each diagonal element $d'_{rr} = 0$. The terms in
\eqref{eq:Psi-on-VS-second} with second tensor factor of the form
$ba = qab$ yield
$$
(\bar d'_{jl} z_l \ox b)(z_i \ox a) 
= p(z_i \ox a)(\bar d'_{jl} z_l \ox b)
\implies \bar d'_{jl} (q z_l z_i - p z_i z_l) = 0  \mot{for} i \neq j.
$$
With $l = i$, $d'_{ji}(q - p) = 0$ for $i \neq j$. Since 
$D'\ovl{D'} = -qI$ makes $D' \neq 0$, that forces $p = q$. 

But then $\bar d'_{jl} (z_l z_i - z_i z_l) = 0$ for $l \neq i$. Since 
$p < 1$, the relation \eqref{eq:VS-odd-first} implies that
$d'_{jl} = 0$ for $l \neq i$, so $d'_{ji}$ is the only possible 
nonzero entry in column~$d'_{\8i}$. But $i,j$ are arbitrary subject
only to $i \neq j$, so all columns of~$D'$ would be null -- since each
has at least three entries -- which contradicts $D' \neq 0$. Thus the
case $A = I$ is ruled out.

Quite similarly, if $A = 0$ and $C = I$, the same argument leads to
another impasse, on replacing $A$ by~$C$ and $B'$ by~$D'$.
\end{proof}

\begin{prop} 
\label{pr:single-one-or-zero}
If $m > 1$, there is no algebra coaction $\Psi$ of $\sA(\rSU_q(2))$ on
$\sA(\bS^{2m+1}_p)$ if the diagonal of $A$ has either \emph{(a):}
a~single $0$~entry, or \emph{(b):} a~single $1$~entry.
\end{prop}

\begin{proof}
\Adit{a}
Suppose first that $a_{00} = 0$ and $a_{ii} = 1$ for $i = 1,\dots,m$.
Then \eqref{eq:flensing} shows that:
\begin{itemize}[nosep]
\item
the only nonzero row of $B$ is $b_{0\8}$ and the only nonzero
column of~$D$ is $d_{\80}$;
\item
the first diagonal elements of $B$ and $D$ vanish:
$b_{00} = d_{00} = 0$;
\item
the only possible nonzero entry of $B'$ is the first diagonal
element~$b'_{00}$;
\item
the first row and column of $D'$ are null:
$d'_{0\8} = d'_{\80} = 0$;
\end{itemize}
and, by Lemma~\ref{lm:bare-columns}, in each column of every matrix at
most one entry is nonzero.

For $z_0$ and $z_l$ with $l \geq 1$, the formula
\eqref{eq:general-Psi} now simplifies to:
\begin{align}
\Psi(z_0) 
&= z_0 \ox a^* + b'_{00} z^*_0 \ox b + \sum_{k=1}^m b_{0k} z_k \ox b,
\nonumber \\
\Psi(z_l) 
&= z_l \ox a + d_{l0} z_0 \ox b^* + \sum_{k=1}^m d'_{lk} z^*_k \ox b^*.
\label{eq:single-zero} 
\end{align}

Applying $\Psi$ to $z_l z_0 = p z_0 z_l$ and comparing terms with
second tensor factor $a^*b^* = q b^*a^*$,
$$
\bigl( d_{l0} z_0 + \tsum_k d'_{lk} z^*_k \bigr) z_0 
= pq z_0 \bigl( d_{l0} z_0 + \tsum_k d'_{lk} z^*_k \bigr)
$$
shows that $d_{l0} (1 - pq) z_0^2 = 0$ for all $l \geq 1$, yielding
$D = 0$; and since $c_{00} = 1$, the relation $-qC = B'\ovl{B'}$
now gives $-q = |b'_{00}|^2$, which cannot happen; so this case is
ruled out.

Since $m > 1$, similar arguments apply if some other $a_{jj} = 0$ and
$a_{ii} = 1$ for $i \neq j$.

\medskip

\Adit{b}
Next consider the case $a_{00} = 1$, $a_{jj} = 0$ for $j \geq 1$.
Relations \eqref{eq:flensing} now show that:
\begin{itemize}[nosep]
\item
the only nonzero row of $D$ is $d_{0\8}$ and the only nonzero column 
of~$B$ is $b_{\80}$;
\item
the first diagonal elements of $B$ and $D$ vanish:
$b_{00} = d_{00} = 0$;
\item
the only possible nonzero entry of $D'$ is the first diagonal
element~$d'_{00}$;
\item
the first row and column of $B'$ are null: 
$b'_{0\8} = b'_{\80} = 0$;
\end{itemize}
and, by Lemma~\ref{lm:bare-columns}, in each column of every matrix
at most one entry is nonzero. 

For $z_0$ and $z_l$ with $l \geq 1$, \eqref{eq:general-Psi} now
simplifies to:
\begin{align}
\Psi(z_0) 
&= z_0 \ox a + d'_{00} z^*_0 \ox b^* + \sum_{k=1}^m d_{0k} z_k \ox b^*,
\nonumber \\
\Psi(z_l) 
&= z_l \ox a^* + b_{l0} z_0 \ox b + \sum_{k=1}^m b'_{lk} z^*_k \ox b.
\label{eq:single-one} 
\end{align}

Applying $\Psi$ to $z_l z_0 = p z_0 z_l$ as before,
the terms with second tensor factor $ba = qab$ yield
$$
q\bigl( b_{l0} z_0 + \tsum_k b'_{lk} z^*_k \bigr) z_0 
= p z_0 \bigl( b_{l0} z_0 + \tsum_k b'_{lk} z^*_k \bigr).
$$
In particular, $b'_{lk} p(1 - q) z_0 z^*_k = 0$ leads to $b'_{lk} = 0$
for $k,l \geq 1$ and hence $B' = 0$ in this case, too. And now
$b_{l0} (p - q) z_0^2 = 0$ with $b_{\80} \neq 0$ shows that $q = p$.

Comparing instead the terms with second tensor factor
$a^* b^* = q b^* a^*$, we now get
$$
d'_{00} (q - q^2) z_l z^*_0 
+ \tsum_k d_{0k} q(z_l z_k - z_k z_l) = 0.
$$ 
Since $q - q^2 \neq 0$, this gives $d'_{00} = 0$, thus $D' = 0$. 
We are left with $\sum_k q d_{0k} (z_l z_k - z_k z_l) = 0$, implying
$$
\sum_{k<l} d_{0k} (q - 1)z_k z_l + \sum_{k>l} d_{0k} (1 - q)z_l z_k
= 0,
$$
so that $d_{0k} = 0$ for $k \neq l$. Mindful that 
$k,l \in \{1,\dots,m\}$ and $m > 1$, this entails $d_{0\8} = 0$
and hence $D = 0$, contradicting the relation $-qA = DB$. So this
case, too, is ruled out.
\end{proof}

We now take a closer look at the results of Lemma~\ref{lm:scalar-A},
applicable to the algebras $\sA(\bS^{2m+1}_p)$.

\begin{lema} 
\label{lm:new-ceros}
If $a_{jj} = 0$, then
\begin{enumerate}[noitemsep]
\item 
$b'_{lj} = 0$ for $l \neq j$, i.e., only the diagonal entry~$b'_{jj}$
of column $b'_{\8j}$ may be nonzero.
\item 
$b'_{lk} = 0$ if $k \neq j$ and $l \neq j$.
\end{enumerate}
Moreover, if $a_{ii} = 1$, then
\begin{enumerate}[noitemsep, resume]
\item 
$d'_{ki} = 0$ for $k \neq i$, i.e., only the diagonal entry~$d'_{ii}$
of column $d'_{\8i}$ may be nonzero.
\item 
$d'_{kl} = 0$ for $k \neq i$ and $l \neq i$.
\end{enumerate}
\end{lema}

\begin{proof}
\Adit{a,b}
In the equality \eqref{eq:Psi-on-VS-second}, replacing the index $i$
by~$l$, and noting that $c_{jj} = 1$, those terms with second tensor
factor $ba = qab$ yield
\begin{align*}
& \sum_{k=0}^m z_j^* (b_{lk} z_k + b'_{lk} z^*_k)
+ q(\bar b_{jk} z_k^* + \bar d'_{jk} z_k) a_{ll} z_l
\\
&\quad = \sum_{k=0}^m pq (b_{lk} z_k + b'_{lk} z^*_k) z_j^*
+ p a_{ll} z_l (\bar d_{jk} z_k^* + \bar d'_{jk} z_k). 
\end{align*}
Collecting terms with first tensor factor $z_j^* z_k^*$ or
$z_k^* z_j^*$, we get $b'_{lk} z_j^* z^*_k = pq b'_{lk} z^*_k z_j^*$. If
$k = j$, since $0 < pq < 1$, this implies $b'_{lj} = 0$ for $l \neq j$.

If $k < j$, \eqref{eq:VS-odd-first} gives
$p b'_{lk} z_k^* z^*_j = pq b'_{lk} z^*_k z_j^*$. Since $0 < q < 1$,
that yields $b'_{lk} = 0$ for $l \neq j$.

Otherwise, if $k > j$, we get
$b'_{ik} z_j^* z^*_k = p^2 q b'_{lk} z^*_j z_k^*$, and again
$b'_{lk} = 0$ for $l \neq j$.

\medskip

\Adit{c,d}
Similar arguments, taking \eqref{eq:Psi-on-VS-second} with 
the index $j$ by~$k$ and noting that $a_{ii} = 1$, collecting only the 
terms with second tensor factor $ba = qab$ and first tensor factor
$z_k z_i$ or $z_i z_k$, lead to the stated conclusions.
\end{proof}

The previous lemma allows to supplement \eqref{eq:flensing} with two
more relations.

\begin{corl} 
\label{cr:flensing-rows}
The following relations on the \emph{rows} of $B'$ and $D'$ hold:
\begin{align}
a_{jj} = 0 &\implies b'_{i\8} = 0 \mot{for} i \neq j;
\nonumber \\
a_{ii} = 1 &\implies d'_{j\8} = 0 \mot{for} j \neq i.
\label{eq:flensing-bis} 
\end{align}
\end{corl}

\begin{proof}
Assume $a_{jj} = 0$; then relations (a) and (b) of
Lemma~\ref{lm:new-ceros} hold. In any column $b'_{\8k}$ of the
matrix $B'$, nonzero entries may only occur in row~$b'_{j\8}$:
case~(a) if $k = j$, case~(b) if $k \neq j$. Therefore, all the other
rows of~$B'$ are zero; that gives the first implication
of~\eqref{eq:flensing-bis}.

The second implication follows in the same way from
Lemma~\ref{lm:new-ceros}(c,d).
\end{proof}

\begin{corl} 
\label{cr:new-ceros}
If two diagonal entries of $A$ are~$0$, then $B' = 0$; and if two
diagonal entries of $A$ are~$1$, then $D' = 0$.
\end{corl}

\begin{prop} 
\label{pr:two-of-each}
If $m > 1$, the case where $A$ is not a scalar matrix does not yield a
$*$-algebra coaction \eqref{eq:general-Psi} of $\sA(\rSU_q(2))$
on~$\sA$.
\end{prop}

\begin{proof}
By permuting rows and columns if necessary, we may suppose that
$\diag(A) = (0,\dots,0,1,\dots,1)$, with $r$ entries of~$0$ followed
by $s$ entries of~$1$, where $r \geq 1$, $s \geq 1$ and
$r + s = m + 1$. The relations \eqref{eq:flensing} show that the
matrices $B$ and $D$ can be written in block form as
$$
B = \twobytwo{0}{B_*}{0}{0}, \qquad
D = \twobytwo{0}{0}{D_*}{0},
$$
where $B_*$ is an $r \x s$ matrix and $D_*$ is an $s \x r$ matrix. 

The cases $r = 1$ and $s = 1$ have been ruled out by
Proposition~\ref{pr:single-one-or-zero}; thus we may suppose that
$m \geq 3$ and that $r \geq 2$ and $s \geq 2$. Then, by
Corollary~\ref{cr:new-ceros}, already $B' = D' = 0$; and the relation
$BD + B'\ovl{B'} = -qC$ reduces to $B_* D_* = -q I_r$. In particular,
no row of $B_*$ and no column of $D_*$ is null.

In the same way, the relation $DB + D'\ovl{D'} = -qA$ reduces to
$D_* B_* = -q I_s$, so each column of $B_*$ has exactly one nonzero
entry, and $D_*$ has no null rows. Both of these reduced relations can
only hold simultaneously if $B_*$ and $D_*$ are square matrices, i.e.,
if $r = s$. Thus $m$ must be odd, and $r = s = \half(m + 1)$. Then
$B_*$, $D_*$ are invertible $r \x r$ matrices, whose nonzero entries
are those of (mutually inverse) permutation matrices.

The relations \eqref{eq:better-Psi}, for $j < r$ and $i \geq r$,
reduce~to:
\begin{alignat}{2}
\Psi(z_j) &= z_j \ox a^* + \tsum_h b_{jh} z_h \ox b,  \qquad &
\Psi(z_i) &= z_i \ox a + \tsum_k d_{ik} z_k \ox b^*,
\nonumber \\
\Psi(z_j^*) &= z_j^* \ox a + \tsum_h \bar b_{jh} z_h^* \ox b^*, &
\Psi(z_i^*) &= z_i^* \ox a^* + \tsum_k \bar d_{ik} z_k^* \ox b,
\label{eq:better-Psi-again} 
\end{alignat}
where $b_{jh} \neq 0$ and $d_{ik} \neq 0$ for only one value of
$h \geq r$ and $k < r$, respectively. Choose $k,l$ so that
$d_{m-1,k} \neq 0$ and $d_{ml} \neq 0$.

Applying $\Psi$ to the relation 
$z_{m-1}^* z_{m-1} = z_{m-1} z_{m-1}^* + (1 - p^2) z_m z_m^*$
from~\eqref{eq:VS-odd-third}, we get
\begin{align*}
& \bigl(z_{m-1}^* \ox a^* 
+ \bar d_{m-1,k} z_k^* \ox b\bigr)
\bigl(z_{m-1} \ox a + d_{m-1,k}\, z_k \ox b^*\bigr)  
\\
&\quad = \bigl(z_{m-1} \ox a + d_{m-1,k}\, z_k \ox b^*\bigr)
\bigl(z_{m-1}^* \ox a^* 
+ \bar d_{m-1,k} z_k^* \ox b\bigr)
\\
&\qquad + (1 - p^2) \bigl(z_m \ox a + d_{ml}\, z_l \ox b^*\bigr)
\bigl(z_m^* \ox a^* + \bar d_{ml} z_l^* \ox b\bigr).
\end{align*}
Comparing terms with second tensor factor $ba = q ab$, we obtain
\begin{equation}
q \bar d_{m-1,k} z_k^* z_{m-1}
= \bar d_{m-1,k} z_{m-1} z_k^*
+ (1 - p^2) \bar d_{ml}  z_m z_l^*.
\label{eq:bad-juju} 
\end{equation}
Since $k \leq r - 1 < m - 1$, this would imply
$$
(pq - 1) \bar d_{m-1,k} z_{m-1} z_k^* 
= (1 - p^2) \bar d_{ml} z_m z_l^*
$$
which is false, since $z_{m-1} z_k^*$ is not a nonzero multiple of
$z_m z_l^*$. (See Remark~\ref{rk:independent-quadratics}.)
\end{proof}

Propositions \ref{pr:scalar-A}, \ref{pr:single-one-or-zero}
and~\ref{pr:two-of-each} exhaust the possibilities for a ``linear''
coaction of type \eqref{eq:general-Psi}, with the sole exception of
the coproduct of $\sA(\rSU_q(2))$. The following result ensues.

\begin{prop} 
\label{pr:sad-news}
If $m > 1$, there is \emph{no first-degree coaction} of
$\sA(\rSU_q(2))$ on $\sA(\bS^{2m+1}_p)$.
\end{prop}

In particular, the $7$-sphere of Vaksman and Soibelman does not admit
a $*$-algebra coaction of first degree. We compare this with the
quaternionic $7$-sphere in Section~\ref{sec:BL-reconsidered} below.
But, since there is indeed such a right coaction of $\sA(\rSU_q(2))$
on $\sO(\bS^7_q)$, we come to the main result.

\begin{thm} 
\label{th:non-isomorphic}
The algebras $\sA(\bS^7_q)$ and $\sO(\bS^7_q)$ are non-isomorphic,
for any $q$ with $0 < q < 1$.
\end{thm}

\begin{proof}
Assume that there is an isomorphism of $*$-algebras
$\al \: \sO(\bS^7_q) \to \sA(\bS^7_q)$, and let 
$y_j := \al(x_j) \in \sA(\bS^7_q)$. Then $y_0,y_1,y_2,y_3$ generate 
$\sA(\bS^7_q)$ as a $*$-algebra, while satisfying the commutation
relations \eqref{eq:BL-S7-commutation} and the sphere relation
\eqref{eq:BL-sphere}, on replacing each $x_j$ with~$y_j$.

From the commutation relations, it is clear that $\sA(\bS^7_q)$ is 
linearly spanned by monomials of the form
\begin{subequations}
\label{eq:new-generators} 
\begin{equation}
y_0^j y_1^k y_2^l y_3^m (y_3^*)^n (y_2^*)^r (y_1^*)^s (y_0^*)^t,
\label{eq:new-generators-raw} 
\end{equation}
with $j,k,l,m,n,r,s,t \in \bN$. The sphere relation allows to reduce
this by replacing 
\begin{equation}
y_3 y_3^* \mapsto 1 - y_0 y_0^* - y_1 y_1^* - y_2 y_2^*,
\label{eq:new-generators-reduced} 
\end{equation}
\end{subequations}
and thus we may assume that $m = 0$ or $n = 0$ (or both). With this
proviso that $mn = 0$, and calling $j + k + l + m + n + r + s + t$ the
\emph{length} of the monomial \eqref{eq:new-generators-raw}, these
monomials form a linear generating set for $\sA(\bS^7_q)$. (A more
detailed verification of this claim can be made, exactly as in the
proof of Proposition~3.3 in~\cite{DAndrea24}.) Using the diamond
lemma \cite{Bergman78} with the sole
relation~\eqref{eq:new-generators-reduced}, one can check that they
form a linear basis for $\sA(\bS^7_q)$.

In a polynomial expression (a linear combination of such monomials),
whose ``degree'' is the length of its longest monomial, this degree is
additive under multiplication. For instance,
\begin{align*}
(y_1 y_3^*)(y_0 y_3) 
&= q^{-1/2} y_1 y_0 y_3^* y_3 - q^{-1}(1 - q^2) y_1 y_2 y_1^* y_3
\\
&= q^{-3/2} y_0 y_1 - q^{-5/2} y_0^2 y_1 y_0^*
- q^{-3/2} y_0 y_1^2 y_1^* 
\\
&\qquad + q^{1/2}(1 - 2q^2) y_0 y_1 y_2 y_2^*
- q^{-3/2}(1 - q^2) y_1 y_2 y_3 y_1^*.
\end{align*}
In particular, since all commutation relations
\eqref{eq:BL-S7-commutation} are purely quadratic, there are
odd-degree polynomials $p_k$ such that $z_k = p_k(y_i, y_j^*)$.

In the same way, using the commutation relations among the $z_i$ and
the sphere relation given by \eqref{eq:VS-odd}
and~\eqref{eq:VS-sphere}, the elements $y_i$ may be expressed as
polynomials of type \eqref{eq:new-generators} with every letter~$y$
replaced by~$z$ -- that, indeed, is the content of
\cite[Prop.~3.3]{DAndrea24}. Hence, $y_i = q_i(z_r, z_s^*)$, for some
odd-degree polynomials~$q_i$. On substituting these polynomials in
$z_k = p_k(y_i, y_j^*)$, we recover each $z_k$ as a polynomial of type
\eqref{eq:new-generators} in the original generators, which must be
just the monomial $z_k$ itself. Therefore, none of the intermediate
polynomials $p_k$ or~$q_i$ can have degree greater than~$1$.

The upshot is that each new generator $y_i$ of $\sA(\bS^7_q)$ is a
\emph{linear} combination of the original generators $z_r$ and
$z^*_s$. Now if $\dl_R$ is the right coaction \eqref{eq:BL-coaction}
above, the map
$$
\Psi := (\al \ox \id) \circ \dl_R \circ \al^{-1}
: \sA(\bS^7_q) \to \sA(\bS^7_q) \ox \sA(\rSU_q(2))
$$
would define a \emph{first-degree} coaction of $\sA(\rSU_q(2))$ on
$\sA(\bS^7_q)$, which has been ruled out by 
Proposition~\ref{pr:sad-news}. In conclusion, no such isomorphism
$\al$ can exist.
\end{proof}

\subsection{Coactions on noncommutative $3$-spheres} 
\label{ssc:three-sphere-cases}

When $m = 1$, some of the previous no-go Propositions do not hold, and
it is necessary to examine this case on its own. As already noted, the
Vaksman--Soibelman $3$-sphere algebra $\sA(\bS^3_q)$ is just 
$\sA(\rSU_q(2))$, since the relations \eqref{eq:VS-odd} and
\eqref{eq:VS-odd-sphere} coincide with \eqref{eq:SUq2-relations} when
$m = 1$, with the identifications $z_0 = a$, $z_1 = b$. But, for
clarity, we keep the notations $z_0$, $z_1$ in the first tensor
factor.

Consider the case $A = I$ of a coaction of $\sA(\rSU_q(2))$ on
$\sA(\bS^3_p)$. Corollary~\ref{cr:scalar-A} shows that the only 
nonzero coefficient matrices are
$$
A = \twobytwo{1}{0}{0}{1} \word{and} 
D' = \twobytwo{0}{d'_{01}}{d'_{10}}{0} 
\word{with} d'_{01} \bar d'_{10} = -q.
$$
Thus \eqref{eq:scalar-A} simplifies to
$$
\Psi(z_0) = z_0 \ox a + d'_{01} z_1^* \ox b^*, \qquad
\Psi(z_1) = z_1 \ox a + d'_{10} z_0^* \ox b^*
$$
and their conjugate relations. The first part of the proof of
Proposition~\ref{pr:scalar-A} shows that $d'_{10}(q - p) = 0$, and 
here also $p = q$. 

Write $d'_{10} =: \om r$ with $|\om| = 1$, $ r > 0$; and then
$d'_{01} = -\om q/r$. Now the relations \eqref{eq:SUq2-relations} show
that
$$
\Psi(z_1) \Psi(z_0) - q \Psi(z_0) \Psi(z_1) 
= \om q(1 - q^2)(r - 1/r) z_1^* z_1 \ox ab^*,
$$
which vanishes if and only if $r = 1$. That leaves us with the right 
coaction
\begin{align}
\Psi(z_0) &= z_0 \ox a - \om q\,z_1^* \ox b^*,
\nonumber \\
\Psi(z_1) &= z_1 \ox a - \om\,z_0^* \ox b^*, \word{with} |\om| = 1.
\label{eq:SUq2-coaction-one} 
\end{align}

Another coaction of $\sA(\rSU_q(2))$ on itself comes from taking
$A = 0$ and $C = I$ and exchanging $D'$ with $B'$, see
Corollary~\ref{cr:scalar-A} again. Now putting $b'_{10} = \om r$ and
$b'_{01} = -\om q/r$, we again find $r = 1$, and obtain
\begin{align}
\Psi(z_0) &= z_0 \ox a^* - \om q\,z_1^* \ox b,
\nonumber \\
\Psi(z_1) &= z_1 \ox a^* - \om\,z_0^* \ox b,
\label{eq:SUq2-coaction-two} 
\end{align}
again with a free parameter $\om$ such that $|\om| = 1$.

\goodbreak 

Next consider the case $a_{00} = 0$, $a_{11} = 1$. From the proof of
Proposition~\ref{pr:single-one-or-zero}(a), the only (possibly)
nonzero entries of the coefficient matrices are $a_{11} = c_{00} = 1$,
$b_{01}$, $d_{10}$, $b'_{00}$ and $d'_{11}$. The formulas
\eqref{eq:single-zero} simplify to
\begin{align*}
\Psi(z_0) &= z_0 \ox a^* + b'_{00} z^*_0 \ox b + b_{01} z_1 \ox b,
\\
\Psi(z_1) &= z_1 \ox a + d_{10} z_0 \ox b^* + d'_{11} z_1^* \ox b^*.
\end{align*}
On computing $\Psi(z_1) \Psi(z_0) - p \Psi(z_0) \Psi(z_1)$, the 
coefficient of $z_0^2 \ox b^* a^*$ is $d_{10}(1 - pq) = 0$ which 
forces $d_{10} = 0$ and hence $D = 0$. And then $-qC = B'\ovl{B'}$
shows that $-q = |b'_{00}|^2$, impossible; so this case is also
ruled out when $m = 1$.

\medskip

The last case to consider is $a_{00} = 1$, $a_{11} = 0$. 
From the proof of
Proposition~\ref{pr:single-one-or-zero}(b), the only (possibly)
nonzero entries of the coefficient matrices are $a_{00} = c_{11} = 1$,
$b_{10}$, $d_{01}$, $b'_{11}$ and $d'_{00}$. The formulas
\eqref{eq:single-one} simplify to
\begin{align*}
\Psi(z_0) &= z_0 \ox a + d'_{00} z_0^* \ox b^* + d_{01} z_1 \ox b^*,
\\
\Psi(z_1) &= z_1 \ox a^* + b_{10} z_0 \ox b + b'_{11} z_1^* \ox b.
\end{align*}
On computing $\Psi(z_1) \Psi(z_0) - p \Psi(z_0) \Psi(z_1)$, the 
coefficient of $z_0 z_1^* \ox ab$ is $b'_{11}p(q - 1) = 0$ which 
forces $b'_{11} = 0$ and hence $B' = 0$. And then $-qC = BD$
shows that $-q = b_{10} d_{01}$.

Put $b_{10} =: \bar\om r$ with $|\om| = 1$, $r > 0$ as before;
now with
$$
\Psi(z_0) = z_0 \ox a + d_{01} z_1 \ox b^*, \qquad
\Psi(z_1^*) = z_1^* \ox a + \bar b_{10} z_0^* \ox b^*,
$$
we obtain
$$
\Psi(z_1^*) \Psi(z_0) - q \Psi(z_0) \Psi(z_1^*)
= \om q(1 - q^2)(r - 1/r) z_1^* z_1 \ox ab^*,
$$
which again forces $r = 1$. Since $z_0 = a$, $z_1 = b$, the upshot is 
that
\begin{align}
\Psi(a) &= a \ox a - \om q b \ox b^*,
\nonumber \\
\Psi(b) &= b \ox a^* + \bar\om a \ox b,
\label{eq:SUq2-coaction-three} 
\end{align}
once more with a free parameter $\om \in \bT$. The value $\om = 1$
of course gives the coproduct~\eqref{eq:SUq-two-coproduct}, as
expected.

\section{Application to the quaternionic $7$-sphere} 
\label{sec:BL-reconsidered}

We now examine possible right coactions
$\Psi \: \sO(\bS^7_q) \to \sO(\bS^7_q) \ox \sA(\rSU_q(2))$, to test
the uniqueness of the coaction \eqref{eq:BL-coaction} obtained
by~\cite{BrainL12}. For that, we shall determine all such first-degree
coactions as set out above.

We thus take $\Psi$ to be a coaction given by \eqref{eq:general-Psi},
with coefficients matrices now in $M_4(\bC)$, satisfying $A + C = I$
and $A' + C' = 0$. All relations given in
Table~\ref{tb:coaction-coefficients} are applicable, and Lemmas
\ref{lm:bare-columns} and~\ref{lm:ones-and-zeros} hold; namely, each
matrix column has at most one nonzero entry, and $A$ and $C$ are
idempotent diagonal matrices. The other relations obtained in
Section~\ref{sec:coactions-on-VS} may not be taken for granted.

\goodbreak 

Due to the quaternionic structure of the Hopf fibration derived in
\cite{LandiPR06} and \cite{BrainL12}, the generators $x_j$ of the 
$*$-algebra $\sO(\bS^7_q)$ are organized in two pairs with a certain
amount of symmetry $(x_0,x_1) \otto (x_2,x_3)$, see their relations
\eqref{eq:BL-S7-quaternionic} and~\eqref{eq:BL-sphere} above; although
this partial symmetry is broken by the commutation relations
in~\eqref{eq:BL-S7-braided}. We shall now make the Ansatz that
$\Psi$~conserves this symmetry: namely, in the notation 
of~\eqref{eq:general-Psi},
$$
\Psi(x_0,x_1,x_2,x_3) = \Psi(x_2,x_3,x_0,x_1).
$$
The right coaction $\dl_R$, declared in \cite[Lemma~3.4]{BrainL12} and
herein written as~\eqref{eq:BL-coaction}, indeed has this structure.

Writing each matrix $M = A$, $B$, etc.~of \eqref{eq:general-Psi} in
block form $\twobytwo{M_{11}}{M_{12}}{M_{21}}{M_{22}}$, this entails
$$
M_{11} \binom{x_0}{x_1} + M_{12} \binom{x_2}{x_3} 
= M_{21} \binom{x_2}{x_3} + M_{22} \binom{x_0}{x_1},
$$
or more compactly, these $2 \x 2$ blocks satisfy
\begin{equation}
M_{11} = M_{22}  \word{and}  M_{12} = M_{21}.
\label{eq:matrix-matching} 
\end{equation}

Under this Ansatz \eqref{eq:matrix-matching}, the conclusion of
Lemma~\ref{lm:A-prime-vanishes} still holds, although its previous
proof relied on the commutation relations \eqref{eq:VS-odd-second}
in the algebra $\sA(\bS^{2m+1}_p)$. However, the relations for 
$x_3^* x_0$ and $x_3^* x_1$ in~\eqref{eq:BL-S7-braided} are
complicated by some extra terms.

\begin{lema} 
\label{lm:A-prime-still-vanishes}
For $\sA = \sO(\bS^7_q)$, the conditions \eqref{eq:matrix-matching}
imply $A' = C' = 0$.
\end{lema}

\begin{proof}
As in the first part of the proof of Lemma~\ref{lm:A-prime-vanishes}, 
we again compare coefficients of the terms $x_k^2 \ox a^2$ and 
$x_k^{*2} \ox a^2$ on applying $\Psi$ to the relations of type
$x_i^* x_j = p_{ij} x_j x_i^*$ in \eqref{eq:BL-S7-commutation}, to
arrive at
$$
a_{ik} \bar a'_{jk} = 0  \word{and} c_{ik} \bar a'_{jk} = 0
\word{for} i \neq j \mot{in} \{0,1,2\} \text{ (only)}.
$$
Using $A + C = I$, that shows that the leading $3 \x 3$ block of~$A'$
is diagonal. Invoking \eqref{eq:matrix-matching} for $M = A'$, this 
gives
$$
A' = \begin{pmatrix}
a'_{00} & 0 & 0 & 0 \\
0 & a'_{11} & 0 & a'_{13} \\
0 & 0 & a'_{00} & 0 \\
0 & a'_{13} & 0 & a'_{11} \end{pmatrix}.
$$

Now $A^2 - A'\ovl{A'} = A$ (Table~\ref{tb:coaction-coefficients}) 
and $A^2 = A$ (Lemma~\ref{lm:ones-and-zeros}) imply $A'\ovl{A'} = 0$.
Since $A'$ here is symmetric, $A'\ovl{A'} = A'(A')^*$ is positive 
semidefinite, and $A'(A')^* = 0$ implies $A' = 0$. 

Lastly, $C' = -A' = 0$ also.
\end{proof}

As an immediate corollary, the simpler
Table~\ref{tb:simpler-coefficients} is also applicable for coactions 
on $\sO(\bS^7_q)$ satisfying \eqref{eq:matrix-matching}. And then
both Lemma~\ref{lm:scalar-A} and Corollary~\ref{cr:scalar-A}
are also valid in this instance.

\medskip

The symmetry $A_{11} = A_{22}$ implies that the only possible 
diagonals of~$A$ are $(1,0,1,0)$, $(0,1,0,1)$ and the scalar diagonals
$(1,1,1,1)$ and $(0,0,0,0)$. We proceed to discard the first two 
options.

\begin{lema} 
\label{lm:impossible-BL}
For $\sA = \sO(\bS^7_q)$, $\diag(A) = (1,0,1,0)$ is not a valid
configuration.
\end{lema}

\begin{proof} 
By Lemma~\ref{lm:scalar-A} and~\eqref{eq:matrix-matching},
\begin{alignat*}{2}
B &= \begin{pmatrix}
0 & 0 & 0 & 0 \\
b_{10} & 0 & b_{12} & 0 \\
0 & 0 & 0 & 0 \\
b_{12} & 0 & b_{10} & 0 \end{pmatrix}, \qquad &
D &= \begin{pmatrix}
0 & d_{01} & 0 & d_{03} \\
0 & 0 & 0 & 0 \\
0 & d_{03} & 0 & d_{01} \\
0 & 0 & 0 & 0 \end{pmatrix},
\\
\shortintertext{and also}
B' &= \begin{pmatrix}
0 & 0 & 0 & 0 \\
0 & b'_{11} & 0 & b'_{13} \\
0 & 0 & 0 & 0 \\
0 & b'_{13} & 0 & b'_{11} \end{pmatrix}, \qquad &
D' &= \begin{pmatrix}
d'_{00} & 0 & d'_{02} & 0 \\
0 & 0 & 0 & 0 \\
d'_{02} & 0 & d'_{00} & 0 \\
0 & 0 & 0 & 0 \end{pmatrix},
\end{alignat*}
with at most one nonzero entry in each column; which entails
$$
b_{10} b_{12} = d_{01} d_{03} = b'_{11} b'_{13} = d'_{00} d'_{02} = 0.
$$
Relations (f) and (h) of Table~\ref{tb:simpler-coefficients} then
imply
$$
b_{k\8} \. d_{\8k} + b'_{k\8} \. \bar b'_{\8k} = -q
\mot{for} k = 1,3;
\qquad
d_{l\8} \. b_{\8l} + d'_{l\8} \. \bar d'_{\8l} = -q 
\mot{for} l = 0,2.
$$
Immediately, $B = 0$ or $D = 0$ cannot hold, since
$\|b'_{k\8}\|^2 = -q$ or $\|d'_{l\8}\|^2 = -q$ would follow. 
Moreover, $b_{10} = 0 = d_{03}$ or $b_{12} = 0 = d_{01}$ are not
acceptable for the same reason.

\medskip

If $b_{10} = 0 = d_{01}$ while $b_{12} \neq 0 \neq d_{03}$, then
\eqref{eq:general-Psi} becomes
\begin{align*}
\Psi(x_0) &= x_0 \ox a + d_{03} x_3 \ox b^* 
+ (d'_{00} x_0^* + d'_{02} x_2^*) \ox b^*,
\\
\Psi(x_1) &= x_1 \ox a^* + b_{12} x_2 \ox b
+ (b'_{11} x_1^* + b'_{13} x_3^*) \ox b,
\\
\Psi(x_2) &= x_2 \ox a + d_{03} x_1 \ox b^*
+ (d'_{02} x_0^* + d'_{00} x_2^*) \ox b^*,
\\
\Psi(x_3) &= x_3 \ox a^* + b_{12} x_0 \ox b 
+ (b'_{13} x_1^* + b'_{11} x_3^*) \ox b,
\end{align*}
with
$$
b_{12} d_{03} + |b'_{11}|^2 + |b'_{13}|^2 = -q 
= b_{12} d_{03} + |d'_{00}|^2 + |d'_{02}|^2  \word{and}
b'_{11} b'_{13} = 0 = d'_{00} d'_{02}.
$$

Applying $\Psi$ to the relation 
$x_3^* x_0 = q^{-1/2} x_0 x_3^* + (q - q^{-1}) x_2 x_1^*$ and then 
comparing terms with first tensor factor $x_1 x_1^*$ or $x_3 x_3^*$,
and second tensor factor $b^*a = qab^*$, we obtain
$$
d_{03} x_3^* x_3 
= q^{1/2} d_{03} x_3 x_3^* - (1 - q^2) d_{03} x_1 x_1^*,
$$
where \eqref{eq:BL-S7-quaternionic} has also been used. But that would
imply $d_{03} = 0$, contrary to assumption.

The alternative $b_{21} = 0 = d_{03}$ with $b_{10} \neq 0 \neq d_{01}$
will likewise lead to an impasse.
\end{proof}

A similar calculation shows that $\diag(A) = (0,1,0,1)$ cannot hold, 
either.

\begin{prop} 
\label{pr:coaction-BL}
There are two families of first-degree right coactions of
$\sA(\rSU_q(2))$ on $\sO(\bS^7_q)$, given by the formulas
\begin{subequations}
\label{eq:coaction-BL} 
\begin{align}
A = I \word{and} \Psi(\xx) &= \xx \ox a + \om T(\xx^*) \ox b^*,
\label{eq:coaction-BL-Aone} 
\\
A = 0 \word{and} \Psi'(\xx) &= \xx \ox a^* + \om T(\xx^*) \ox b,
\label{eq:coaction-BL-Azero} 
\end{align}
where $\om \in \bT$ and $T \in M_4(\bC)$ is of the form 
\eqref{eq:matrix-matching}, with
\begin{equation}
T_{11} = T_{22} = \twobytwo{0}{1}{-q}{0} \word{and}
T_{12} = T_{21} = \twobytwo{0}{0}{0}{0}.
\label{eq:coaction-BL-coeffs} 
\end{equation}
\end{subequations}
\end{prop}

\begin{proof}
\Adit{a}
Suppose first that $A = I$. From Corollary~\ref{cr:scalar-A}, the 
only other nonzero coefficient matrix is $D'$, which satisfies
$D'\ovl{D'} = -q I$ and has zeros on the diagonal. Since
$d'_{l\8} \. d'_{\8l} = -q$ for each~$l$, no row or column of~$D'$
can be null. By Lemma~\ref{lm:bare-columns}, each column -- and now
also each row -- of $D'$ has exactly one (non-diagonal) nonzero entry.
Applying \eqref{eq:matrix-matching}, we obtain
$$
D'_{11} = D'_{22} = \twobytwo{0}{d'_{01}}{d'_{10}}{0}, \qquad
D'_{12} = D'_{21} = \twobytwo{d'_{02}}{d'_{03}}{d'_{12}}{d'_{13}}.
$$
Suppose that $d'_{01} = 0$; then in row $d'_{0\8}$, either 
$d'_{03} \neq 0$ or $d'_{02} \neq 0$, but not both. If
$d'_{03} \neq 0$ so that $d'_{02} = 0 = d'_{13}$, we get
$$
\Psi(x_0) = x_0 \ox a + d'_{03} x_3^* \ox b^*,  \qquad
\Psi(x_2) = x_2 \ox a + d'_{03} x_1^* \ox b^*.
$$
From $\Psi(x_2) \Psi(x_0) = q^{1/2} \Psi(x_0) \Psi(x_2)$, the terms
with second tensor factor $b^* a = q a b^*$ yield
$$
d'_{03} (1 - q^{5/2}) x_2 x_3^* = d'_{03} (q^{1/2} - q^2) x_0 x_1^*,
$$
impossible since $d'_{03} \neq 0$. On the other hand, if
$d'_{02} \neq 0$ so that $d'_{03} = 0 = d'_{12}$, we find that
$$
\Psi(x_0) = x_0 \ox a + d'_{02} x_2^* \ox b^*,  \qquad
\Psi(x_2) = x_2 \ox a + d'_{02} x_0^* \ox b^*.
$$
The same calculation now yields
$$
d'_{02} (1 - q^{3/2}) x_2 x_2^* 
= d'_{02} (q^{1/2} - q) x_0 x_0^*
$$
again impossible since $d'_{02} \neq 0$. 

The upshot is that $d'_{01} \neq 0$; and indeed $d'_{10} \neq 0$, too 
-- together with $D'_{12} = 0$ -- because $D'\ovl{D'} = -q I$ implies
$d'_{01} \bar d'_{10} = -q$.

Write $d'_{01} =: \om r \in \bC$ with $|\om| = 1$, $r > 0$; then
$d'_{10} = -\om q/r$. Now \eqref{eq:general-Psi} reduces to 
\begin{alignat}{2}
\Psi(x_0) &= x_0 \ox a + \om r x_1^* \ox b^*, &
\Psi(x_2) &= x_2 \ox a + \om r x_3^* \ox b^*,
\nonumber \\
\Psi(x_1) &= x_1 \ox a - \om(q/r) x_0^* \ox b^*, \qquad &
\Psi(x_3) &= x_3 \ox a - \om(q/r) x_2^* \ox b^*.
\label{eq:BL-coaction-amplified} 
\end{alignat}

Comparing terms with second tensor factor $a^* b^* = q b^* a^*$ on
applying $\Psi$ to the sphere relation \eqref{eq:BL-sphere}, we deduce
that
$$
\om q^2 (r - 1/r) (x_1^* x_0^* + x_3^*  x_2^*) = 0,
$$
which implies $|d'_{01}| = r = 1$; without any condition on~$\om$.

\medskip

\Adit{b}
Suppose now that $A = 0$, or equivalently, $C = I$. Again
by Corollary~\ref{cr:scalar-A}, we see that $B = D = D' = 0$ and
$B'\ovl{B'} = -q I$, implying $b'_{l\8} \. b'_{\8l} = -q$ for
each~$l$, so no row or column of~$B'$ can be null; each column and row
has exactly one (non-diagonal) nonzero entry. Thus,
$$
B'_{11} = B'_{22} = \twobytwo{0}{b'_{01}}{b'_{10}}{0}, \qquad
B'_{12} = B'_{21} = \twobytwo{b'_{02}}{b'_{03}}{b'_{12}}{b'_{13}}.
$$
Suppose now that $b'_{01} = 0$ and $b'_{03} \neq 0$,
so that $b'_{02} = 0 = b'_{13}$; thereby,
$$
\Psi'(x_0) = x_0 \ox a^* + b'_{03} x_3^* \ox b,  \qquad
\Psi'(x_2) = x_2 \ox a^* + b'_{03} x_1^* \ox b.
$$
Again from $\Psi'(x_2) \Psi'(x_0) = q^{1/2} \Psi'(x_0) \Psi'(x_2)$,
using terms with second tensor factor $a^* b = q b a^*$, we find
$$
b'_{03} q(1 - q^{1/2}) (x_2 x_3^* - x_0 x_1^*) = 0,
$$
contradicting $b'_{03} \neq 0$. If instead $b'_{01} = 0$ and
$b'_{02} \neq 0$, so that $b'_{03} = 0 = b'_{12}$, we arrive at
$$
\Psi'(x_0) = x_0 \ox a^* + b'_{02} x_2^* \ox b,  \qquad
\Psi'(x_2) = x_2 \ox a^* + b'_{02} x_0^* \ox b,
$$
leading to
$$
b'_{02} (1 - q^{3/2}) x_0 x_0^* = b'_{02} (q^{1/2} - q) x_2 x_2^*
$$
contrary to $b'_{02} \neq 0$. As a result, $b'_{01} \neq 0$.
And as before $b'_{10} \neq 0$, together with $B'_{12} = 0$, since
$B'\ovl{B'} = -q I$ implies $b'_{01} \bar b'_{10} = -q$.

In this case, we put $b'_{01} =: \om r \in \bC$ with $|\om| = 1$
and $r > 0$; and $b'_{10} = -\om(q/r)$. The coaction now looks like
\begin{alignat}{2}
\Psi'(x_0) &= x_0 \ox a^* + \om r x_1^* \ox b, &
\Psi'(x_2) &= x_2 \ox a^* + \om r x_3^* \ox b,
\nonumber \\
\Psi'(x_1) &= x_1 \ox a^* - \om(q/r) x_0^* \ox b, \qquad &
\Psi'(x_3) &= x_3 \ox a^* - \om(q/r) x_2^* \ox b,
\label{eq:BL-coaction-extended} 
\end{alignat}
and just as before we find that $r = 1$.

Both cases of \eqref{eq:coaction-BL} have now been verified.
\end{proof}

It is clear that the coaction \eqref{eq:BL-coaction} obtained by Brain
and Landi in \cite{BrainL12} is the case $\om = 1$
of~\eqref{eq:BL-coaction-amplified}. The new feature of our analysis
is the second coaction family \eqref{eq:BL-coaction-extended} (putting
$r = 1$ in both formulas).

\begin{prop} 
\label{pr:coalgebra-invariants}
The coaction-invariant subalgebra $\sB$ of $\sO(\bS^7_q)$ for the
right coaction \eqref{eq:coaction-BL-Aone} does not depend on~$\om$.
It is the $*$-algebra generated by the elements $X_0$, $X_1$, $X_2$,
where
\begin{align}
X_0 &:= 2\bigl( q^2 x_0 x_0^* + x_1^* x_1 \bigr) - 1 
= 2\bigl( x_0 x_0^* + x_1 x_1^* \bigr) - 1 = X_0^*,
\nonumber \\
X_1 &:= 2\bigl( q^2 x_0 x_2^* + x_1^* x_3 \bigr), \qquad
X_2 := 2\bigl( qx_0 x_3^* - x_1^* x_2 \bigr),
\label{eq:four-sphere-generators} 
\end{align}
subject to the relations that $X_0$ is central, and that
\begin{alignat}{2}
X_1 X_2 &= X_2 X_1, \qquad &
X_2^* X_2 &= q^{-1} X_2 X_2^*,
\nonumber \\
X_1^* X_2 &= q^{-1} X_2 X_1^*, \qquad &
X_1^* X_1 &= q X_1 X_1^* - q^{-2}(1 - q^2)\,X_2^* X_2,
\label{eq:BL-S4-commutation} 
\end{alignat}
plus the sphere relation:
\begin{equation}
q X_1 X_1^* + q^{-1} X_2 X_2^* + X_0^2 = 1.
\label{eq:four-sphere-relation} 
\end{equation}
These relations allow us to identify $\sB =: \sA(\bS^4_q)$ as a
noncommutative $4$-sphere.
\end{prop}

\begin{proof}
The proof of the algebra relations \eqref{eq:BL-S4-commutation}
and \eqref{eq:four-sphere-relation} is a lengthy but straightforward 
computation, which we defer to 
Appendix~\ref{app:coalgebra-invariants}. 

It is easy to check that the $X_i$ of 
\eqref{eq:four-sphere-generators} are coaction invariants. For 
instance,
\begin{align*}
\Psi(\half(X_0 + 1))
&= q^2 (x_0 \ox a + \om x_1^* \ox b^*)
(x_0^* \ox a^* + \bar\om x_1 \ox b) 
\\
&\qquad + (x_1^* \ox a^* - q \bar\om x_0 \ox b)
(x_1 \ox a - q \om x_0^* \ox b^*)
\\
&= q^2 x_0 x_0^* \ox (a a^* + b b^*)
+ x_1^* x_1 \ox (q^2 b^* b + a^* a)  
\\
&= q^2 x_0 x_0^* \ox 1 + x_1^* x_1 \ox 1 = \half(X_0 + 1) \ox 1.
\end{align*}
Likewise, $\Psi(\half X_1) = q^2 x_0 x_2^* \ox 1 + x_1^* x_3 \ox 1$
and $\Psi(\half X_2) = q x_0 x_3^* \ox 1 - x_1^* x_2 \ox 1$ by similar
calculations.
\end{proof}

\begin{remk} 
\label{rk:coalgebra-invariants}
The commutation relations \eqref{eq:BL-S4-commutation} and
\eqref{eq:four-sphere-relation} of the noncommutative $4$-sphere $\sB$
follow directly from the defining relations
\eqref{eq:BL-S7-commutation} and~\eqref{eq:BL-sphere} of the
noncommutative $7$-sphere $\sO(\bS^7_q)$. Now, apart from notational
differences ($q \otto q^{-1}$ and $X_i \otto x_i$) the $7$-sphere
relations and the coaction invariants are those given in the
paper~\cite{BrainL12}. In that article, the
generators~\eqref{eq:four-sphere-generators} of~$\sB$ appear as
entries of a matrix projector~$\pp$. However, the equation
$\pp^2 = \pp$ leads to relations that differ from those given here,
see \cite[Eq.~(3.18)]{BrainL12}, so it is worthwhile to obtain the
relations directly, in Appendix~\ref{app:coalgebra-invariants}.
\end{remk}

One can also show that the structure obtained so far is that of a 
\textit{noncommutative principal bundle}; namely, that the canonical 
map $\chi \: \sO(\bS^7_q) \ox_{\sB} \sO(\bS^7_q)
\to \sO(\bS^7_q) \ox \sA(\rSU_q(2))$ is bijective. As noted, for
instance, in~\cite[Sect.~12.1]{Sontz15}, following~\cite{Schneider90},
it is enough to show that the \emph{lifted} canonical map
$$
\tilde\chi \: \sO(\bS^7_q) \ox \sO(\bS^7_q)
\to \sO(\bS^7_q) \ox \sA(\rSU_q(2)) : x \ox y \mapsto x\,\Psi(y)
$$
is surjective. Thus, for instance,
\begin{align*}
\MoveEqLeft{\tilde\chi
(q^2 x_0^* \ox x_0 + x_1^* \ox x_1 + q^2 x_2^* \ox x_2 + x_3^* \ox x_3)}
\\
&= (q^2 x_0^* x_0 + x_1^* x_1 + q^2 x_2^* x_2 + x_3^* x_3) \ox a
\\
&\qquad + (q^2 \om x_0^* x_1^* - q \om x_1^*x_0^* + q^2 \om x_2^* x_3^*
- q \om x_3^* x_2^*) \ox b^* = 1 \ox a,
\end{align*}
using \eqref{eq:BL-S7-quaternionic} and~\eqref{eq:BL-sphere}; and 
similarly,
\begin{align*}
\tilde\chi(x_0 \ox x_0^* + x_1 \ox x_1^* 
+ x_2 \ox x_2^* + x_3 \ox x_3^*) &= 1 \ox a^*,
\\
\tilde\chi(x_1^* \ox x_0^* - q x_0^* \ox x_1^* 
+ x_3^* \ox x_2^* - q x_2^* \ox x_3^*) &=  \bar \om \ox b,
\\
\tilde\chi(x_0 \ox x_1 - q x_1 \ox x_0 
+ x_2 \ox x_3 - q x_3 \ox x_2) &= -q \om \ox b^*.
\end{align*}
Since $\bar\om$ and $-q\om$ are scalars, the image of~$\tilde\chi$
contains $1 \ox c$ for each algebra generator~$c$ of $\sA(\rSU_q(2))$;
and since $\Psi$ is an algebra coaction, the surjectivity
of~$\tilde\chi$ follows.

To conclude, we briefly consider the other coaction $\Psi'$ given 
by~\eqref{eq:coaction-BL-Azero}.

\begin{lema} 
\label{lm:coalgebra-invariants}
The coaction-invariant subalgebra $\sB'$ of $\sO(\bS^7_q)$ for the
right coaction \eqref{eq:coaction-BL-Azero} does not depend on~$\om$.
It is the $*$-algebra generated by 
\begin{align*}
Y_0 &:= 2\bigl( q^2 x_0^* x_0 + x_1 x_1^* \bigr) - 1 
= 2\bigl( x_0^* x_0 + x_1^* x_1 \bigr) - 1 = Y_0^*,
\\
Y_1 &:= 2\bigl( q^2 x_0^* x_2 + x_1 x_3^* \bigr), \qquad
Y_2 := 2\bigl( qx_0^* x_3 - x_1 x_2^* \bigr).
\end{align*}
\end{lema}

The coaction-invariance of $\sB'$ and the surjectivity of 
$\tilde\chi' : x \ox y \mapsto x\,\Psi'(y)$ can be checked on algebra 
generators, just as for $\sB$ and~$\tilde\chi$. The commutation 
relations among the $Y_i$ and~$Y_i^*$ are less evident, so it is not 
clear whether $\sB'$ could also be called a noncommutative $4$-sphere.

\appendix

\section{The noncommutative $4$-sphere relations} 
\label{app:coalgebra-invariants}

To complete the proof of Proposition~\ref{pr:coalgebra-invariants},
the relations \eqref{eq:BL-S4-commutation} and
\eqref{eq:four-sphere-relation} must be established. This is just a
direct calculation, but since the results disagree with
\cite[Eq.~(3.18)]{BrainL12}, we give the details.

First, from the algebra relations \eqref{eq:BL-S7-commutation} we
deduce a few simple cubic relations:
\begin{alignat}{3}
x_0 x_0^* x_1 &= q^2\, x_1 x_0^* x_0, \qquad &
x_0 x_0^* x_2 &= x_2 x_0^* x_0, &
x_0 x_2^* x_1 &= q\, x_1 x_2^* x_0,
\nonumber \\
x_2 x_0^* x_1 &= q\, x_1 x_0^* x_2, &
x_2 x_2^* x_0 &= q\, x_0 x_2^* x_2, \qquad &
x_2 x_2^* x_1 &= q\, x_1 x_2^* x_2.
\label{eq:cubic-simple} 
\end{alignat}
A few other cubic relations have more terms on the right-hand side:
\begin{align}
x_2 x_1^* x_1 &= x_1 x_1^* x_2 + (1 - q^2) x_0 x_0^* x_2,
\nonumber \\
x_3 x_0^* x_0 &= x_0 x_0^* x_3 + q^{-3/2}(1 - q^2)\,x_0 x_2^* x_1,
\nonumber \\
x_3 x_1^* x_2 &= x_2 x_1^* x_3 - q^{-1/2}(1 - q^2) x_2 x_2^* x_0,
\nonumber \\
x_3 x_1^* x_1 &= x_1 x_1^* x_3
+ (1 - q^2)(x_3 x_0^* x_0 - q^{-1/2}\,x_1 x_2^* x_0),
\nonumber \\
x_0 x_3^* x_3 &= q\,x_3 x_3^* x_0
+ (1 - q^2) (x_0 x_2^* x_2 + q^{-1/2}\,x_3 x_1^* x_2).
\label{eq:cubic-tricky} 
\end{align}

We list the algebra generators of $\sB$ in a convenient form:
\begin{alignat*}{2}
X_0 &= 2 (x_0 x_0^* + x_1 x_1^*) - 1 \quad
&= X_0^* &= 1 - 2(x_2 x_2^* + x_3 x_3^*),
\\
X_1 &= 2(q^{-1/2} x_3 x_1^* + x_0 x_2^*), &
X_1^* &= 2(q^{-1/2} x_1 x_3^* + x_2 x_0^*),
\\
X_2 &= 2(qx_0 x_3^* - q^{1/2} x_2 x_1^*), &
X_2^* &= 2 (q x_3 x_0^* -  q^{1/2} x_1 x_2^*).
\end{alignat*}

The centrality of $X_0$ and the relations \eqref{eq:BL-S4-commutation}
and~\eqref{eq:four-sphere-relation} can now be verified directly,
using \eqref{eq:cubic-simple} and~\eqref{eq:cubic-tricky}, as follows.

\begin{enumerate}
\item 
To show that $X_0$ commutes with $X_1$ (and with $X_1^*$), we first
note a few products of monomials $x_i x_j^*$:
\begin{align*}
x_0 x_0^* \, x_0 x_2^* &= x_0 x_2^* \, x_0 x_0^*,
\\
x_1 x_1^* \, x_0 x_2^* &= x_0 x_2^* \, x_1 x_1^*,
\\
x_0 x_0^* \, x_3 x_1^* &= q^{-2}\,x_3 x_1^* \, x_0 x_0^*
- q^{-3/2}(1 - q^2) x_0 x_2^* x_1 x_1^*,
\\
x_1 x_1^* \, x_3 x_1^* &= x_3 x_1^* \, x_1 x_1^*
- (1 - q^2) x_3 x_0^* x_0 x_1^*
+ q^{-1/2}(1 - q^2) x_1 x_2^* x_0 x_1^*.
\end{align*}
Then
\begin{align*}
\MoveEqLeft{
\quarter [X_0, X_1] = [x_0 x_0^*, x_0 x_2^*] + [x_1 x_1^*, x_0 x_2^*]
+ q^{-1/2} [x_0 x_0^*, x_3 x_1^*] + q^{-1/2} [x_1 x_1^*, x_3 x_1^*]}
\\
&= q^{-2}(1 - q^2) \bigl( q^{-1/2}\,x_3 x_1^* x_0 x_0^*
- x_0 x_2^* x_1 x_1^* - q^{3/2}\,x_3 x_0^* x_0 x_1^*
+ q\,x_1 x_2^* x_0 x_1^* \bigr) = 0.
\end{align*}
Here and in what follows we have used some of the formulas
\eqref{eq:cubic-simple}, e.g., to get the cancellations in the last
line.

\item 
To show that $X_0$ commutes with $X_2$ (and with $X_2^*$), we proceed 
similarly:
\begin{align*}
x_0 x_0^* \, x_2 x_1^* &= q^{-2}\,x_2 x_1^* \, x_0 x_0^*,
\\
x_1 x_1^* \, x_2 x_1^* &= x_2 x_1^* \, x_1 x_1^*
- (1 - q^2) x_2 x_0^* x_0 x_1^*,
\\
x_0 x_0^* \, x_0 x_3^* &= x_0 x_3^* \, x_0 x_0^*
+ q^{-3/2}(1 - q^2) x_0 x_1^* x_2 x_0^*,
\\
x_1 x_1^* \, x_0 x_3^* &= x_0 x_3^* \, x_1 x_1^*
- q^{-1/2}(1 - q^2) x_0 x_0^* x_2 x_1^*,
\end{align*}
which implies
\begin{align*}
\MoveEqLeft{
\quarter [X_0, X_2]
= -q^{-1/2} [x_0 x_0^*, x_2 x_1^*] - q^{-1/2} [x_1 x_1^*, x_2 x_1^*]
+ q [x_0 x_0^*, x_0 x_3^*] + q [x_1 x_1^*, x_0 x_3^*]}
\\
&= -q^{-3/2}(1 - q^2) \bigl( x_2 x_1^* x_0 x_0^*
- q^2\,x_2 x_0^* x_0 x_1^* - q\,x_0 x_1^* x_2 x_0^*
+ q^2\,x_0 x_0^* x_2 x_1^* \bigr) = 0.
\end{align*}
It follows that $X_0$ is central in~$\sB$.

\item 
To show that $X_1$ and $X_2$ commute:
\begin{align*}
x_0 x_2^* \, x_2 x_1^* &= q^{-2}\,x_2 x_1^* \, x_0 x_2^*,
\\
x_3 x_1^* \, x_2 x_1^* &= x_2 x_1^* \, x_3 x_1^*
- q^{-1/2}(1 - q^2) x_2 x_2^* x_0 x_1^*,
\\
x_0 x_2^* \, x_0 x_3^* &= x_0 x_3^* \, x_0 x_2^*
+ q^{-3/2}(1 - q^2) x_0 x_1^* x_2 x_2^*,
\\
x_3 x_1^* \, x_0 x_3^* &= x_0 x_3^* \, x_3 x_1^*
- (1 - q^2) x_0 x_2^* x_2 x_1^*,
\end{align*}
and therefore
\begin{align*}
\MoveEqLeft{
\quarter [X_1, X_2]
= -q^{1/2} [x_0 x_2^*, x_2 x_1^*] - [x_3 x_1^*, x_2 x_1^*]
+ q [x_0 x_2^*, x_0 x_3^*] + q^{1/2} [x_3 x_1^*, x_0 x_3^*]}
\\
&= -q^{-3/2}(1 - q^2) \bigl( x_2 x_1^* x_0 x_2^*
- q\,x_2 x_2^* x_0 x_1^* - q\,x_0 x_1^* x_2 x_2^*
+ q^2\,x_0 x_2^* x_2 x_1^* \bigr) = 0.
\end{align*}

\item 
For the commutation relation of $X_2^*$ and $X_2$, we need
\begin{align*}
x_1 x_2^* \, x_2 x_1^* &= q^{-1}\,x_2 x_1^* \, x_1 x_2^*
- q^{-1}(1 - q^2) x_0 x_0^* x_2 x_2^*,
\\
x_1 x_2^* \, x_0 x_3^* &= q^{-1}\,x_0 x_3^* \, x_1 x_2^*
- q^{-3/2}(1 - q^2) x_0 x_0^* x_2 x_2^*,
\\
x_3 x_0^* \, x_2 x_1^* &= q^{-1}\,x_2 x_1^* \, x_3 x_0^*
- q^{-3/2}(1 - q^2) x_2 x_2^* x_0 x_0^*,
\\
x_3 x_0^* \, x_0 x_3^* &= q^{-1}\,x_0 x_3^* \, x_3 x_0^*
- q^{-1}(1 - q^2) x_0 x_2^* x_2 x_0^*,
\end{align*}
whereby
$$
\quarter(X_2^* X_2 - q^{-1} X_2 X_2^*)
= (1 - q^2) \bigl( x_0 x_0^* x_2 x_2^* - x_0 x_0^* x_2 x_2^*
+ x_2 x_2^* x_0 x_0^* - q\,x_0 x_2^* x_2 x_0^* \bigr) = 0.
$$

\item 
For the commutation relation of $X_1^*$ and $X_2$, we need
\begin{align*}
x_2 x_0^* \, x_2 x_1^* &= q^{-1}\,x_2 x_1^* \, x_2 x_0^*,
\\
x_2 x_0^* \, x_0 x_3^* &= q^{-1}\,x_0 x_3^* \, x_2 x_0^*,
\\
x_1 x_3^* \, x_2 x_1^* &= q\,x_2 x_1^* \, x_1 x_3^*
- q(1 - q^2) x_2 x_0^* x_0 x_3^*,
\\
x_1 x_3^* \, x_0 x_3^* &= q^{-1}\,x_0 x_3^* \, x_1 x_3^*
- q^{-3/2}(1 - q^2) x_0 x_0^* x_2 x_3^*
- q^{-3/2}(1 - q^2) x_1 x_1^* x_2 x_3^*
\end{align*}
and thus
\begin{align*}
\quarter(X_1^* X_2 - q^{-1} X_2 X_1^*)
&= q^{-1}(1 - q^2) \bigl( x_2 x_1^* x_1 x_3^* 
+ q^2\,x_2 x_0^* x_0 x_3^* - x_0 x_0^* x_2 x_3^*
- x_1 x_1^* x_2 x_3^* \bigr)
\\
&= -q(1 - q^2) \bigl( x_0 x_0^* x_2 x_3^*
- x_2 x_0^* x_0 x_3^* \bigr) = 0.
\end{align*}

\item 
For the commutation of $X_1^*$ and $X_1$, we appeal to
\begin{align*}
x_2 x_0^* \, x_0 x_2^* &= q\,x_0 x_2^* \, x_2 x_0^*,
\\
x_2 x_0^* \, x_3 x_1^* &= q^{-1}\,x_3 x_1^* \, x_2 x_0^*
- q^{-3/2}(1 - q^2) x_2 x_2^* x_1 x_1^*,
\\
x_1 x_3^* \, x_0 x_2^* &= q^{-1}\,x_0 x_2^* \, x_1 x_3^*
- q^{-3/2}(1 - q^2) x_1 x_1^* x_2 x_2^*,
\\
x_1 x_3^* \, x_3 x_1^* &= q\,x_3 x_1^* \, x_1 x_3^*
- q(1 - q^2) x_3 x_0^* x_0 x_3^* + (1 - q^2) x_1 x_2^* x_2 x_1^*.
\end{align*}
The second and third of these can be rewritten as
\begin{align*}
x_2 x_0^* \, x_3 x_1^* &= q\,x_3 x_1^* \, x_2 x_0^*
+ q^{-1}(1 - q^2) x_3 x_1^* x_2 x_0^*
- q^{-3/2}(1 - q^2) x_2 x_2^* x_1 x_1^*,
\\
x_1 x_3^* \, x_0 x_2^* &= q\,x_0 x_2^* \, x_1 x_3^*
+ q^{-1}(1 - q^2) x_0 x_2^* x_1 x_3^*
- q^{-3/2}(1 - q^2) x_1 x_1^* x_2 x_2^*.
\end{align*}
Therefore,
\begin{align*}
\quarter(X_1^* X_1 - q X_1 X_1^*) = q^{-2}(1 - q^2) \bigl( 
& q^{1/2}\,x_3 x_1^* x_2 x_0^*
- x_2 x_2^* x_1 x_1^* + q^{1/2}\,x_0 x_2^* x_1 x_3^*
\\
& - x_1 x_1^* x_2 x_2^* - q^2\,x_3 x_0^* x_0 x_3^*
+ q\,x_1 x_2^* x_2 x_1^* \bigr).
\end{align*}
On the other hand,
$$
\quarter X_2^* X_2 
= q^2 x_3 x_0^* x_0 x_3^* - q^{1/2} x_3 x_1^* x_2 x_0^*
- q^{1/2} x_0 x_2^* x_1 x_3^* + q x_1 x_2^* x_2 x_1^*.
$$
Hence
\begin{align*}
\MoveEqLeft{
\quarter \bigl( X_1^* X_1 - q X_1 X_1^* 
+ q^{-2}(1 - q^2) X_2^* X_2 \bigr)}
\\
&= q^{-2}(1 - q^2) \bigl( 2q\,x_1 x_2^* x_2 x_1^*
- x_2 x_2^* x_1 x_1^* - x_1 x_1^* x_2 x_2^* \bigr) = 0,
\end{align*}

\item 
Finally, to verify the sphere relation 
\eqref{eq:four-sphere-relation}, it is enough to compute
\begin{align*}
\quarter X_1 X_1^*
&= q^{-1}\,x_3 x_1^* x_1 x_3^* + q^{-1/2}\,x_0 x_2^* x_1 x_3^*
+ q^{-1/2}\,x_3 x_1^* x_2 x_0^* + x_0 x_2^* x_2 x_0^*,
\\
\quarter X_2 X_2^*
&= q^2\,x_0 x_3^* x_3 x_0^* - q^{3/2}\,x_0 x_3^* x_1 x_2^*
- q^{3/2}\,x_2 x_1^* x_3 x_0^* + q\,x_2 x_1^* x_1 x_2^*,
\\
\quarter(1 - X_0^2)
&= x_0 x_0^* x_2 x_2^* + x_0 x_0^* x_3 x_3^*
+ x_1 x_1^* x_2 x_2^* + x_1 x_1^* x_3 x_3^*.
\end{align*}
Then, using \eqref{eq:cubic-simple} and~\eqref{eq:cubic-tricky}:
\begin{align*}
& \quarter \bigl( q\,X_1 X_1^* + q^{-1} X_2 X_2^* + X_0^2 - 1 \bigr)
\\
&\quad = (x_3 x_1^* x_1 x_3^* - x_1 x_1^* x_3 x_3^*)
+ (q\,x_0 x_3^* x_3 x_0^* - x_0 x_0^* x_3 x_3^*)
+ q^{1/2}(x_3 x_1^* x_2 x_0^* - x_2 x_1^* x_3 x_0^*)
\\
&\qquad + q^{1/2}(x_0 x_2^* x_1 x_3^* - x_0 x_3^* x_1 x_2^*)
+ (q\,x_0 x_2^* x_2 x_0^* - x_0 x_0^* x_2 x_2^*)
+ (x_2 x_1^* x_1 x_2^* - x_1 x_1^* x_2 x_2^*)
\\
&\quad = (1 - q^2) \bigl\{ x_3 x_0^* x_0 x_3^*
- q^{-1/2}\,x_1 x_2^* x_0 x_3^*
- x_0 x_0^* x_3 x_3^* + q\,x_0 x_2^* x_2 x_0^*
+ q^{1/2}\,x_0 x_2^* x_1 x_3^*
\\
&\hspace*{6em}
- x_2 x_2^* x_0 x_0^* - x_0 x_0^* x_2 x_2^* + x_0 x_0^* x_2 x_2^*
\bigr\}
\\
&\quad = (1 - q^2) \bigl\{ (x_3 x_0^* x_0 - x_0 x_0^* x_3) x_3^*
- q^{-1/2}(x_1 x_2^* x_0 - q\,x_0 x_2^* x_1)x_3^* \bigr\}
\\
&\quad = q^{-3/2}(1 - q^2)^2 (x_0 x_2^* x_1 x_3^*
- q\,x_1 x_2^* x_0 x_3^*) = 0.
\end{align*}
\end{enumerate}

\section{Addendum: is there a coaction by $\sA(\rSU(2))$?} 
\label{app:no-SUtwo-coaction}

In Section~\ref{sec:coactions-on-VS} it was shown that there are no 
first-degree right coactions of $\sA(\rSU_q(2))$ on the 
Vaksman--Soibelman spheres $\sA(\bS^{2m+1}_p)$ for $m \geq 2$,
provided that $0 < p < 1$ and also $0 < q < 1$. One may ask whether 
the same is true when $q = 1$; that is, does the Hopf algebra 
$\sA(\rSU(2))$ of representative functions on the Lie group $\rSU(2)$ 
coact on the Vaksman--Soibelman quantum spheres, always with
$0 < p < 1$? 

The answer is not immediately obvious, since some proofs in that
section do require $q < 1$.

\medskip

We thus revisit the results of Section~\ref{sec:coactions-on-VS} with
$0 < p < 1$, but now $q = 1$.

Table~\ref{tb:simpler-coefficients} is unchanged, apart from
rows (f) and~(h), whose first entries simplify to
$$
BD + B'\ovl{B'} = -C   \word{and}  DB + D'\ovl{D'} = -A.
$$
Lemma~\ref{lm:scalar-A} is likewise unchanged (it depends only on the
other rows). Corollary~\ref{cr:scalar-A} and its proof still hold with
$q = 1$. In Proposition~\ref{pr:scalar-A} with $A = I$,
formula~\eqref{eq:scalar-A} leads to $D' = 0$ instead of $p = q$,
which contradicts the relation $D'\ovl{D'} = -I$ from the previous
Corollary. That rules out a first-degree coaction with $A = I$; and
also with $C = I$ by a similar argument.

In the proof of Proposition~\ref{pr:single-one-or-zero}(b), the
relations \eqref{eq:single-one} yield $B = 0$ instead of $B' = 0$; but
the terms with second tensor factor $bb^* = b^*b$ then give either
$B' = 0$ or $D' = 0$, incompatible with lines (f) or~(h) of
Table~\ref{tb:simpler-coefficients}.

Coming to Lemma~\ref{lm:new-ceros}, its proof remains valid for parts
(a) and~(c) if $q = 1$, but statements (b) and~(d) must be weakened.
The modified statement is as follows.

\begin{lema} 
\label{lm:cleanup}
If $\Psi$ as in~\eqref{eq:first-degree-Psi} defines a right coaction 
of $\sA(\rSU(2))$ on $\sA(\bS^{2m+1}_p)$ with $0 < p < 1$; and if
$a_{jj} = 0$ and $a_{ii} = 1$, then
\begin{enumerate}[nosep]
\item 
$b'_{lj} = 0$ for $l \neq j$, i.e., only the diagonal entry~$b'_{jj}$
of column $b'_{\8j}$ may be nonzero;
\item 
$b'_{lk} = 0$ if $k > j$ and $l \neq j$.
\item 
$d'_{ki} = 0$ for $k \neq i$, i.e., only the diagonal entry~$d'_{ii}$
of column $d'_{\8i}$ may be nonzero;
\item 
$d'_{kl} = 0$ for $k < i$ and $l \neq i$.
\end{enumerate}
\end{lema}

To simplify the supplementary analysis, here we consider only the case
of~$\sA(\bS^7_p)$. Since Propositions \ref{pr:scalar-A}
and~\ref{pr:single-one-or-zero} are valid with $q = 1$, there only
remain the cases where the $4 \x 4$ matrix $A$ has a diagonal that is
some permutation of $(1,1,0,0)$; the same is true of $C = I - A$, of
course.

We analyze the following consequence of \eqref{eq:VS-odd-third}
with $m = 3$:
\begin{equation}
\Psi(z_2^*) \Psi(z_2) 
= \Psi(z_2) \Psi(z_2^*) + (1 - p^2) \Psi(z_3) \Psi(z_3^*).
\label{eq:first-try} 
\end{equation}
The terms $\Psi(z_i)$ and $\Psi(z_j^*)$ may be expanded with the 
formulas~\eqref{eq:better-Psi}. Keeping only the terms with second 
tensor factor $ba = ab$ and monomials of type $z_k z_l^*$ or
$z_r^* z_s$, this delivers:
\begin{align*}
a_{22} \tsum_{k\neq 2} \bar d_{2k} z_k^* z_2
& + c_{22} \tsum_{k\neq 2} b_{2k} z_2^* z_k
= a_{22} \tsum_{k\neq 2} \bar d_{2k} z_2z_k^*
+ c_{22} \tsum_{k\neq 2} b_{2k} z_k z_2^*
\\
&\quad + (1 - p^2) a_{33} \tsum_{l\neq 3} \bar d_{3l} z_3 z_l^* 
+ (1 - p^2) c_{33} \tsum_{l\neq 3} b_{3l} z_l z_3^*.
\end{align*}
Using \eqref{eq:VS-odd-second} and dividing by an overall factor of
$(1 - p)$, that simplifies to
\begin{align}
0 = \sum_{k\neq 2} \bigl( 
a_{22} \bar d_{2k}\, z_2z_k^* + c_{22} b_{2k}\, z_k z_2^* \bigr)
+ (1 + p) \sum_{l\neq 3} \bigl(
a_{33} \bar d_{3l}\, z_3 z_l^* + c_{33} b_{3l}\, z_l z_3^* \bigr).
\label{eq:first-relation} 
\end{align}

It is also useful to consider the related identity:
\begin{equation}
\Psi(z_1^*) \Psi(z_1) = \Psi(z_1) \Psi(z_1^*) + (1 - p^2) \bigl(
\Psi(z_2) \Psi(z_2^*) + \Psi(z_3) \Psi(z_3^*) \bigr).
\label{eq:second-try} 
\end{equation}
Treating this relation in the same way as \eqref{eq:first-try}, we
arrive at:
\begin{align}
0 &= \sum_{r\neq 1} \bigl(
a_{11} \bar d_{1r}\, z_1 z_r^* + c_{11} b_{1r}\, z_r z_1^* \bigr)
\label{eq:second-relation} 
\\
&\quad + (1 + p) \sum_{k\neq 2} \bigl( 
a_{22} \bar d_{2k}\, z_2 z_k^* + c_{22} b_{2k}\, z_k z_2^* \bigr)
+ (1 + p) \sum_{l\neq 3} \bigl(
a_{33} \bar d_{3l}\, z_3 z_l^* + c_{33} b_{3l}\, z_l z_3^* \bigr).
\nonumber 
\end{align}

\medskip

We now consider the possible diagonals of the matrix $A$ on a 
case-by-case basis.

\paragraph{Case 1a}
Take $a_{00} = a_{11} = 1$, $a_{22} = a_{33} = 0$. Then 
$c_{22} = c_{33} = 1$; and from Lemma~\ref{lm:scalar-A},
$b_{0\8} = b_{1\8} = d_{2\8} = d_{3\8} = 0$ and
$b_{\8 2} = b_{\8 3} = d_{\8 0} = d_{\8 1} = 0$. Using
$z_2^* z_k = p z_k z_2^*$ for $k = 0,1$ and dividing by $(1 - p)$,
\eqref{eq:first-relation} reduces to
$$
0 = b_{20} z_0 z_2^* + b_{21} z_1 z_2^*
+ (1 + p) (b_{30} z_0 z_3^* + b_{31} z_1 z_3^*);
$$
and by linear independence of these $z_r z_s^*$ monomials, it follows 
that $B = 0$. By Lemma~\ref{lm:cleanup}(a), 
$(B'\ovl{B'})_{jj} = |b'_{jj}|^2$ for $j = 2,3$. But then
$BD + B'\ovl{B'} = -C$ would imply $|b'_{jj}|^2 = -1$ for $j = 2,3$,
which is absurd. Therefore, the requirement \eqref{eq:first-try} has no
solution in this case.

\paragraph{Case 1b}
Take $a_{00} = a_{11} = 0$, $a_{22} = a_{33} = 1$. This case is
analogous to the previous one, by switching $A \otto C$, $B \otto D$
and $B' \otto D'$ and using Lemma~\ref{lm:cleanup}(c). There is no
possible $\Psi$ satisfying \eqref{eq:first-try} in this case either.

\paragraph{Case 2a}
Take $a_{11} = a_{22} = 1$, $a_{00} = a_{33} = 0$. 
Lemma~\ref{lm:scalar-A} entails
$b_{1\8} = b_{2\8} = d_{0\8} = d_{3\8} = 0$ and
$b_{\8 0} = b_{\8 3} = d_{\8 1} = d_{\8 2} = 0$, while $c_{22} = 0$.
In this case, \eqref{eq:first-relation} gives 
$$
0 = \bar d_{20} z_2 z_0^* + \bar d_{23} z_2 z_3^*
+ (1 + p) (b_{31} z_1 z_3^* + b_{32} z_2 z_3^*);
$$
which implies $d_{20} = 0$, $b_{31} = 0$, and 
$\bar d_{23} = -(1 + p) b_{32}$. Here the relevant matrices are
$$
B = \begin{pmatrix}
0 & b_{01} & b_{02} &0 \\
0 & 0 & 0 & 0 \\
0 & 0 & 0 & 0 \\
0 & 0 & b_{32} & 0 \end{pmatrix},
\quad
D = \begin{pmatrix} 
0 & 0 & 0 & 0 \\
d_{10} & 0 & 0 & d_{13}  \\
0 & 0 &  0 & d_{23}  \\
0 & 0 & 0 & 0 \end{pmatrix},
\quad
B' = \begin{pmatrix} 
b'_{00} & b'_{01} & b'_{02} & 0 \\
0 & 0 & 0 & 0 \\
0 & 0 & 0 & 0 \\
0 & 0 & 0 & 0 \end{pmatrix}.
$$
Indeed, Lemma~\ref{lm:cleanup}(a,b) for $j = 0$ shows that the only
non-null row of $B'$ is $b'_{0\8}$, so that the $(3,3)$ entry of 
the matrix equation $BD + B'\ovl{B'} = -C$ is the relation
$b_{32} d_{23} = -1$.

On the other hand, \eqref{eq:second-relation} now yields
$$
0 = \bar d_{10}\, z_1 z_0^* + \bar d_{13}\, z_1 z_3^*
+ (1 + p) (\bar d_{23} + b_{32})\, z_2 z_3^*,
$$
and in particular, $\bar d_{23} = - b_{32}$, so 
$b_{32} = d_{23} = 0$, contradicting $b_{32} d_{23} = -1$. In summary,
\eqref{eq:first-try} and~\eqref{eq:second-try} have no common 
solution.

\paragraph{Case 2b}
Take $a_{00} = a_{33} = 1$, $a_{11} = a_{22} = 0$. This is again like
Case~2a, by switching $A \otto C$, $B \otto D$ and $B' \otto D'$ 
and now using Lemma~\ref{lm:cleanup}(c,d) with $i = 3$.

\paragraph{Case 3a}
Take $a_{00} = a_{22} = 1$, $a_{11} = a_{33} = 0$. 
Lemma~\ref{lm:scalar-A} entails
$b_{0\8} = b_{2\8} = d_{1\8} = d_{3\8} = 0$ and
$b_{\8 1} = b_{\8 3} = d_{\8 0} = d_{\8 2} = 0$, while $c_{22} = 0$.
Now \eqref{eq:first-relation} gives 
$$
0 = \bar d_{21}\, z_2 z_1^* + (1 + p) b_{30}\, z_0 z_3^*
+ \bigl( \bar d_{23} + (1 + p) b_{32} \bigr)\, z_2 z_3^*
$$
which implies $d_{21} = 0$, $b_{30} = 0$, and 
$\bar d_{23} = -(1 + p) b_{32}$. Here the relevant matrices are
$$
B = \begin{pmatrix}
0 & 0 & 0 & 0 \\
b_{10} & 0 & b_{12} &0 \\
0 & 0 & 0 & 0 \\
0 & 0 & b_{32} & 0 \end{pmatrix},
\quad
D = \begin{pmatrix} 
0 & d_{01} & 0 & d_{03} \\
0 & 0 & 0 & 0 \\
0 & 0 & 0 & d_{23} \\
0 & 0 & 0 & 0 \end{pmatrix},
\quad
B' = \begin{pmatrix} 
b'_{00} & 0 & 0 & 0 \\
b'_{10} & b'_{11} & b'_{12} & 0 \\
b'_{20} & 0 & 0 & 0 \\
b'_{30} & 0 & 0 & 0
\end{pmatrix},
$$
The structure of~$B'$ follows from Lemma~\ref{lm:cleanup}, part~(a) 
for $j = 1,3$ and part~(b) for $j = 1$. Again, the $(3,3)$ entry of 
$BD + B'\ovl{B'} = -C$ is the equation $b_{32} d_{23} = -1$.

This time, \eqref{eq:second-relation} shows that
$$
0 = b_{10}\, z_0 z_1^* + b_{12}\, z_2 z_1^* 
+ (1 + p) \bigl( \bar d_{23} + b_{32} \bigr)\, z_2 z_3^* 
$$
and again $\bar d_{23} = - b_{32}$, so $b_{32} = d_{23} = 0$, 
incompatible with $b_{32} d_{23} = -1$. Once more,
\eqref{eq:first-try} and~\eqref{eq:second-try} have no common 
solution.

\paragraph{Case 3b}
Take $a_{11} = a_{22} = 1$, $a_{00} = a_{33} = 0$. Again this is like
Case~3a, under $A \otto C$, $B \otto D$ and $B' \otto D'$; this time
using Lemma~\ref{lm:cleanup}(c,d) with $i = 2$.

\medskip

The list of cases is complete, and no first-degree coactions have 
been found. The following result has been established.

\begin{prop} 
\label{pr:sadder-news}
On the Vaksman--Soibelman sphere $\sA(\bS^7_p)$, there is 
\emph{no} first-degree right coaction of $\sA(SU(2))$.
\end{prop}


\medskip

\subsection*{Acknowledgments}
We thank Francesco D'Andrea for helpful correspondence, and for
alerting us to the work of Saurabh~\cite{Saurabh17a, Saurabh17b}. The
reviewer's thoughtful comments led to several improvements, and also
to Appendix~\ref{app:no-SUtwo-coaction}. The authors acknowledge the
support of the Centro de Investigación en Matemática Pura y Aplicada
via the research project \#821--C4--162 at the University of Costa
Rica.

%
%
%

\bigskip


\end{document}